\documentclass[11pt,leqno]{article}
\usepackage{graphicx}

\usepackage{amsmath, latexsym, mathtools, amsthm, amsfonts, amssymb}
\usepackage{esint}

\usepackage{cochineal}
\usepackage[cochineal]{newtxmath}

\usepackage[T1]{fontenc}     

\usepackage{array, hyperref, url, mdframed, longtable, tabu, graphicx, mathrsfs, tcolorbox}
\usepackage{geometry}
\usepackage{fullpage}

\hypersetup{
     colorlinks=true,
     linkcolor=blue,
     filecolor=blue,
     citecolor=red,      
     urlcolor=cyan,
     }

\newtheorem{theorem}{Theorem}[section]
\newtheorem{lemma}[theorem]{Lemma}
\newtheorem{proposition}[theorem]{Proposition}
\newtheorem{corollary}[theorem]{Corollary}

\theoremstyle{definition}
\newtheorem{definition}[theorem]{Definition}
\newtheorem{example}[theorem]{Example}
\newtheorem{fact}[theorem]{Fact}

\theoremstyle{remark}
\newtheorem{remark}[theorem]{Remark}

\newtheorem{note}[theorem]{Note}

\newtheorem{claim}[theorem]{Claim}

\numberwithin{equation}{section}

\newcommand{\prend}{$\hfill \Box$}

\newcommand{\vertiii}[1]{{\left\vert\kern-0.25ex\left\vert\kern-0.25ex\left\vert #1 
    \right\vert\kern-0.25ex\right\vert\kern-0.25ex\right\vert}}

\begin{document}


\title{A probabilistic approach to strong natural boundaries}

\author{Stamatis Dostoglou\footnote{Department of Mathematics, University of Missouri, Columbia MO 65211, 
e-mail: \url{dostoglous@missouri.edu} } \and 
Petros Valettas\footnote{Departments of Mathematics and EECS, University of Missouri, Columbia MO 65211, e-mail: \url{valettasp@missouri.edu} } }








\maketitle

\begin{abstract}
	We study the local non-extendability of random power series beyond their disk of convergence. 
	We show that random power series formed by independent coefficients which are asymptotically anti-concentrated admit the circle of
	radius of convergence as strong natural boundary, even in a Nevanlinna sense. Our results extend 
	previous work of Breuer and Simon (2011) for the case of independent coefficients. 
	Our motivation stems from the study of Pad\'e approximants of random power series as a denoising tool.
	\end{abstract}




\tableofcontents


\section{Introduction}

 The circle with radius equal to the radius of convergence of a power series is a natural boundary (NB) if all points of this circle are singular points. 
 The notion of natural boundary goes back at least to Weierstrass \cite{Weier}, 
 and was developed by Kronecker, Goursat, Hadamard, Fatou, Hurwitz, Fabry, Borel, P\'olya, Szeg\"o, among others (see e.g., \cite{Rem-book}). 
The classical results focus on conditions that imply natural boundary ("gap" conditions, for example) and the construction of specific examples. 
Notably, Agmon in \cite[p.\ 297]{Agmon} gives conditions for the power series to be unbounded in any circular sector of its disc of convergence.

More recently, Breuer and Simon introduced and examined in \cite{BS} the notion of a {\it strong natural boundary} (SNB) for power series with
bounded sequence of coefficients: 
The natural boundary of a power series $f(z) = \sum_{k=0}^\infty c_k z^k$ with radius of convergence $r_f$ is a strong natural boundary if for any circular arc $I$ 
\begin{equation} \label{BSforIntro}
   \sup_{0<r<r_f}\int_I 	|f(re^{i\theta})|d\theta = \infty.
\end{equation}
Informed by analogies from spectral theory, Breuer and Simon  use "right limits," cf.\ \cite[p.\ 4904]{BS}, 
akin to Agmon's conditions in \cite{Agmon}, and  identify as a source of strong natural boundaries the existence of right limits that are not "reflectionless," 
see \cite[p.\ 4905]{BS}. They then use their method to show that in several classical results natural boundaries can be strengthened to strong ones. 
We do not insist on the details of right limits and reflectionless right limits, as our methods here are different and extend beyond the class of power series
with uniformly bounded coefficients. 

For power series with random coefficients natural boundaries were studied by Steinhaus, Borel, Kahane, Ryll-Nardzewski, among others, see \cite{Kah}. 
The distilled wisdom is that random power series have natural boundaries, perhaps after a deterministic perturbation. 
More precisely,  power series with independent, symmetric random coefficients always have their circle of convergence as natural boundary, 
see \cite[p.\ 39]{Kah}. The case when the coefficients are not symmetric is covered by \cite{R-N}:

\begin{theorem} [Ryll-Nardzewski, 1953] \label{thm:R-N}
	Let $\{ X_k \}_{k=0}^\infty$ be a sequence of independent random variables and 
	that $F(\omega ; z) = \sum_{k=0}^\infty X_k(\omega) z^k$ has radius of convergence\footnote{For the normalization $r_F=1$ see the discussion at the beginning of Section \ref{S:2-2}.} $r_F=1$ a.s. 
	Then either $F$ almost surely has the unit circle as natural boundary or 
	there exists deterministic $f(z)=\sum_{k=0}^\infty c_k z^k$ with $r_{F-f}>1$ and $F-f$ almost surely has the circle of radius $r_{F-f}$ as natural boundary.
\end{theorem}
In fact, Ryll-Nardzewski shows that the radius of convergence of the symmetrization 
of the powers series is the biggest radius of convergence for all the deterministic perturbations of the power series. 
For generalizations of \cite{R-N} see \cite{Hol} and the references therein.

Breuer and Simon have also examined the existence of strong natural boundaries for power series with uniformly bounded random coefficients, 
and some of their results partly motivate our work here. 
In particular, Breuer and Simon, cf. \cite[Theorem 6.1]{BS}, prove
\begin{theorem}[Breuer, Simon, 2011] \label{thm:BS}
Let $\{X_k\}_{k=0}^\infty$ be independent random variables with 
\begin{align} \label{BSconditions}
     |X_k|\leq M \text{ a.s.  for all } k\geq 1, \; {\rm and} \; \limsup_k{\rm Var}[X_k] >0.
\end{align}
Then $F(z) =\sum_{k = 0}^\infty X_k z^k$ has strong natural boundary almost surely. 
\end{theorem}
As the approach of Breuer and Simon is to find for each arc of the circle of convergence a single "non-reflectionless" right limit, 
they work to produce two right limits, different enough so that if one is reflectionless the other is not. 
That is where their condition $\limsup_k {\rm Var}[X_k] >0 $ becomes instrumental in their approach, 
as positive variance implies separation of values of random variables, cf. \cite[Lemma 6.2]{BS}.

Our motivation comes from studying Pad\'e approximants of a random power series $F$, especially as used for signal denoising, see \cite{DV}, \cite{DPV}.  
As is well known, the $[m,n]$-Pad\'e approximant is a rational function with power series that agrees with $F$ up to the first $m+n+1$ terms. 
As rational functions, Pad\'e approximants are defined on the whole complex plane. 
For power series with independent random coefficients (the pure noise case) it is conjectured that the 
Pad\'e poles cluster around the circle of convergence as the degree of the Pad\'e denominator increases (and similarly for the zeros of the Pad\'e numerator). 
For the kind of random power series $F$ we examine in this article the circle of convergence is a (strong) natural boundary. 
It is therefore reasonable to investigate as much as possible the nature of this boundary and to determine to what extent, if any, 
it forms as the limit of clustering Pad\'e poles.\footnote{Indeed, for certain quasi-analytic functions with natural boundary, Pad\'e approximants converge (in capacity) to the function even beyond the natural boundary \cite{GN}.}
This article consists of our results so far  in the direction of elucidating the nature of the natural boundary for random power series with this direction in mind. We believe that the results are also of independent interest thanks to their connection with the work of Breuer and Simon \cite{BS} (and the more classical work of Ryll-Nardzewski \cite{R-N}).

During our investigation \cite{DV} of the Pad\'e approximation of random power series, 
it became clear that the appearance of (mere) natural boundaries for random power series can be thought of as a result of 
mild anti-concentration properties of the coefficients of the series, see \cite[\S 4.1]{DV}. More precisely, for 
\begin{align}
			Q(\xi, \lambda) 
			:= 
			\sup_{v\in \mathbb R} 
			\mathbb P(v \leq \xi \leq v + \lambda)
\end{align}
the L\'evy concentration function, Theorem 4.2 in \cite{DV} reads:

\begin{theorem}
Let $\{X_k\}_{k=0}^\infty$ be independent random variables with $\sup_k \mathbb{E}|X_k| <\infty$ 
and $\liminf_k Q(X_k, \varepsilon) <1$ for some $\varepsilon>0$. 
Then the power series $F(z)=\sum_{k = 0}^\infty X_k z^k$ a.s. has radius of convergence $r_F=1$, and the unit circle as natural boundary. 
\end{theorem}

Here we follow this path into the territory of strong natural boundaries. 
In particular, in Section \ref{s:3-1} we first present a weaker anti-concentration condition for natural boundaries: 

\begin{proposition} 
	Let $\{X_k\}_{k=0}^\infty$ be  a sequence of independent random variables with 
\begin{equation} \label{unweighted}
	\sum_{k=0}^\infty \left(1- Q(X_k,\varepsilon) \right) =   \infty
\end{equation}
 for some $\varepsilon>0$. Then if the power series $F(z) = \sum_{k = 0}^\infty X_k z^k$ has radius of converge $1$, it a.s. has the unit circle as natural boundary.
\end{proposition}

We then identify, in the same section, a weighted version of the anti-concentration condition \eqref{unweighted} for strong natural boundaries:

\begin{theorem} \label{thm:s-nat-bd}
Let $\{X_k\}_{k=0}^\infty$ be a sequence of independent random variables and 
assume that the random power series $F(z) = \sum_{k=0}^\infty X_k z^k$ has radius of convergence $1$ a.s. If
for some bounded sequence $(t_k)$, $t_k\geq 0$, we have	
	\begin{align} \label{anti_in_Intro}
		\sum_{k=0}^\infty t_k^2 \left( 1 - Q(X_k, t_k) \right) = \infty,
	\end{align}
then the power series $F(z) =\sum_{k = 0}^\infty X_k z^k$ has a.s. the unit circle as strong natural boundary. 
\end{theorem}

This is included in the somewhat more general Theorem \ref{thm:s-nat-bd} below.
The proof relies on the standard Rogozin inequality (Lemma \ref{lem:Rogo}), and a small-ball estimate under mixtures (Lemma \ref{lem:mix-sb}). 

Section \ref{SymmetricSection} applies the main result in the case of symmetric random $X_k$'s. 
Building on the Rademacher case (Proposition \ref{prop:Rad-snb}), 
we characterize when the natural boundary (always present according to \cite{R-N, Kah} for symmetric coefficients) is a strong one:

\begin{theorem} 
	Let $\{X_k\}_{k=0}^\infty$ be a sequence of independent and symmetric random variables and assume that the random
	power series $F(z) = \sum_{k=0}^\infty X_k z^k$ has radius of convergence $1$ a.s. 
	Then the unit circle is strong natural boundary for $F$ if and only if $\sum_{k=0}^\infty |X_k|^2 = \infty$ a.s.
\end{theorem}
This is included in Theorem \ref{thm:snb-sym} below. 

Section \ref {S:bd-case} is where we turn to the case of uniformly bounded $X_k$'s and re-examine the result of Breuer-Simon. 
Here we find that, as variances are now available, the anti-concentration condition \eqref{anti_in_Intro} 
can be relaxed to the condition $\sum_{k=0}^\infty {\rm Var}[X_k]  = \infty$:

\begin{theorem} \label{thm:snb-be}
	Let $\{X_k\}_{k=0}^\infty$ be a sequence of independent random variables with $|X_k| \leq M$ a.s. for all $k$ and $\sum_{k=0}^\infty {\rm Var}[X_k] = \infty$. 
	Then, the random power series $F(z) = \sum_{k=0}^\infty X_k z^k$ has a.s. radius of convergence $1$, and the unit circle as strong natural boundary. 
\end{theorem}
This is contained in Theorem \ref {thm:snb-be}. The proof builds on the proof of Theorem \ref{thm:s-nat-bd} 
with the Berry-Esseen estimate (Lemma \ref{lem:BE} below) replacing Rogozin.  

These results also hold when absolute value in \eqref{BSforIntro} is replaced by any non-decreasing $\psi:[0,\infty) \to [0,\infty)$ with $\psi(t) \to \infty$ as $t \to \infty$. 
To capture the logarithmic function, in Section \ref{S:log-int} we show that the condition
\begin{equation}
     \sup_k Q(X_k, \lambda) \leq b\lambda,
\end{equation}
for all $\lambda >0$ and some $b>0$ suffices, when the power series has radius of convergence $1$, 
to satisfy the logarithmic version of \eqref{BSforIntro} on any arc $I$:
\begin{equation}
     	\sup_{0<r<1} \int_I \log|F(re^{i\theta})| d\theta = \infty.
\end{equation}
This section relies heavily on results from our previous \cite{DV}.

Some implications of a strong natural boundary for the pointwise behavior of the power series on the circle of convergence are included in Section \ref{pointwise}. 
Section 2 gathers some facts and preliminary results that we use from Complex Analysis and Probability Theory.

\section{Preliminaries}

\subsection{Notation} Let $\mathbb D$ denote the open unit disc $\{z\in \mathbb C : |z| <1\}$ and we write $D(z_0,r)$ for the generic open disc
centered at $z_0\in \mathbb C$ with radius $r>0$, i.e., $D(z_0, r) : = \{z\in \mathbb C : |z-z_0| < r \}$. 
We write $\mathbb T=\{ z\in \mathbb C : |z|=1\}$ for the unit circle and $C(z_0,r)$ for the circle centered at $z_0$ having radius $r>0$. For a subset $A$ of $\mathbb C$ 
we write $\overline{A}$ for its closure with respect to the standard topology induced by the modulus on $\mathbb C$. 

The triple $(\Omega, {\mathcal E}, P)$ will stand for a probability space. The random objects will be denoted with uppercase letters, e.g, random variables $X,Y,Z, \ldots$ 
and, consequently, the random functions generated by them, are denoted by $F,G,H, \ldots$. We reserve the lowercase letters for non-random objects (deterministic); 
however, we occasionally use lowercase letters to denote classical random variables (e.g., Bernoulli, Rademacher, etc) and in generic probabilistic facts, as follows.
The mathematical expectation of a random variable $\xi$ is denoted by $\mathbb E[\xi]$ and its variance by ${\rm Var}[\xi] = \mathbb E |\xi- \mathbb E\xi|^2$.

Throughout the text we will make frequent use of universal (numerical) constants, which we shall denoted with $C, c, c_0, \ldots$, and
their value may change from line to line. For two (positive) quantities $Q_1, Q_2$ we write $Q_1 \lesssim Q_2$ if there exists a universal constant $C>0$ so 
that $Q_1 \leq CQ_2$. We write $Q_1 \asymp Q_2$ if $Q_1 \lesssim Q_2$ and $Q_2 \lesssim Q_1$.

\subsection{Complex Analysis fundamentals} 

Recall that a power series $f(z) = \sum_{k=0}^\infty c_k z^k$ with radius of convergence $r_f=1$ admits $z_0 \in \mathbb T$ as a {\it regular}
point, if there exists a neighborhood $D(z_0,\delta)= \{z : |z-z_0|<\delta\}$ such that $f$ has an analytic extension in $\mathbb D \cup D(z_0,\delta)$.
Otherwise the point is called {\it singular.}

For a power series $f$ as above we say that the circle $\mathbb T$ is a {\it natural boundary} for $f$ 
if $f$ cannot be extended to a holomorphic function through any arc of this circle.

A more pathological behavior on the boundary of the radius of convergence of a power series can be described with the notion of {\it strong natural boundary} (SNB). 
This strong singularity, following Breuer and Simon \cite{BS}, is defined as follows: An analytic function $f(z)= \sum_{k=0}^\infty c_k z^k$ on
the disk is said to have {\it strong natural boundary} at $|z|=1$ if for any arc $I=(a,b) \subset (0,2\pi)$ we have
		\begin{align}
			\sup_{0<r<1} \int_I |f(re^{i\theta}) | \, d\theta = \infty. 
		\end{align}
We occasionally refer to this as an {\it $L_1$-strong natural boundary}. With the work of Agmon \cite{Agmon} in mind, 
we shall say that the power series $f$ with radius of convergence $r_f=1$ has the unit circle as an {\it $L^\infty$-strong natural boundary} 
if for any arc $I$
\begin{equation}
     \sup_{0< r < 1, \theta \in I} |f(re^{i\theta})| 
     = \infty	.
\end{equation}

These notions arise naturally in the course of precluding analytic functions to belong to the classical Hardy spaces
$H^1$ and $H^\infty$, even locally. Recall that an analytic function $f:\mathbb D \to \mathbb C$ is in $H^p\equiv H^p (\mathbb D)$, $(0<p<\infty)$ if
	\begin{align}
		\|f\|_{H^p}^p:= \sup_{0<r<1} \left\{ \frac{1}{2\pi }\int_0^{2\pi} |f(re^{i\theta})|^p \, d\theta\right\} <\infty.
	\end{align} 
The $H^\infty$ is defined with respect to the condition $\sup_r \sup_{\theta} |f(re^{i\theta})|<\infty$. 
Replacing $t^p$ by $\log^+ t$ we obtain the Nevanlinna space $N$, i.e., the collection of all analytic functions $f$ on the unit disk $\mathbb D$ for which
	\begin{align}
		\sup_{0<r<1} \left\{ \frac{1}{2\pi} \int_0^{2\pi} \log^+ |f(re^{i\theta})| \, d\theta \right\} <\infty,
	\end{align}
where $\log^+ t =0$ if $0\leq t\leq 1$ and $\log^+ t =\log t$ if $t>1$. It is known that  $H^\infty \subset H^q \subset H^p \subset N$,
for all $0<p<q<\infty$, and that the $\sup_r$ can be replaced by the $\lim_{r\uparrow 1}$ due to the (sub-)harmonicity of the functions
$|f|^p$ and $\log^+ |f|$ (for a proof see e.g. \cite{Du-book}).

Expanding the perspective of \cite{BS}, and taking these definitions into account, one may define several variants of (SNB): For instance, 
we say that $f$ has $|z|=1$ as {\it Nevanlinna}-(SNB) if for any arc $I \subset (0,2\pi)$ satisfies
	\begin{align}
		\sup_{0<r<1} \int_I \log^+ |f(re^{i\theta})| \, d\theta = \infty. 
	\end{align}
In view of the above we conclude the following hierarchy for strong natural boundaries: for $0<p<\infty$
	\begin{align*}
		Nevanlinna-(SNB)\quad \Longrightarrow \quad L_p - (SNB) \quad \Longrightarrow \quad L_\infty - (SNB) .
	\end{align*}	
More generally, one can define $H^\psi$ spaces for {\it test functions} $\psi$. In our context a test function stands for a non-decreasing function 
$\psi: [0,\infty) \to [0,\infty)$ with $\lim_{t\to \infty} \psi(t) = \infty$. 
Thus, a function $f$ belongs to $H^\psi$ if 
	\begin{align}
		\sup_{0<r<1} \int_0^{2\pi} \psi(|f(re^{i\theta})|) \, d\theta< \infty,
	\end{align} and, by analogy, $f$ exhibits
$\psi$-(SNB) if for every $I\subset (0,2 \pi)$ we have 
	\begin{align}
		\sup_{0<r<1} \int_I \psi(|f(re^{i\theta})|) \, d\theta = \infty.
	\end{align} Occasionally, the test functions under consideration will satisfy a sub-additivity 
property as 
	\begin{align} \label{eq:sub-ad}
		\psi(t+s) \leq C [1 +\psi(t) + \psi(s)], \quad t,s\geq 0, \quad (C>0).
	\end{align}

\begin{remark}
	Let us note in passing that the functions $\psi(t)=t^p$ and $\psi(s) = \log^+s$ that arise in classical theory of Hardy spaces, apart from being in the class 
	of test functions defined above, also satisfy a sub-additivity condition. Indeed, for any $z_1, \ldots, z_m\in \mathbb C$ we have
		\begin{align}
			\left| \sum_{i=1}^m z_i \right|^p \leq m^p \sum_{i=1}^m |z_i|^p, \quad 0<p<\infty,
		\end{align}
	and for any $w_1, \ldots, w_m \in \mathbb C$ we have\footnote{W.l.o.g. we may assume that $|w_1| \geq \ldots \geq |w_m| \geq 0$. Taking 
	into account the monotonicity of $t \mapsto \log^+ t$ we may write
		\begin{align}
			\log^+ \left| \sum_i w_i \right| \leq \log^+ \left( \sum_i  |w_i| \right) \leq \log^+ (m|w_1|) \leq \log m + \log^+|w_1| \leq \log m + \sum_{i=1}^m \log^+ |w_i|.
		\end{align} }
		\begin{align} \label{fct:tria-log+}
			\log^+ \left| \sum_{i=1}^m w_i \right| \leq \log m + \sum_{i=1}^m \log^+ |w_i|.
		\end{align}	
\end{remark}

\subsection{Probabilistic toolbox} \label{S:2-2}

We consider random power series of the form $F(z) = \sum_{k=0}^\infty X_kz^k$ with independent coefficients $\{X_k\}_{k=0}^\infty$.
Recall that Kolmogorov's 0-1 Law implies that the radius of convergence $r_F= (\limsup_k |X_k|^{1/k})^{-1}$ is a.s. a constant in $[0,\infty]$.
For meaningfully addressing the question of natural boundary, one must assume that $0<r_F<\infty$ (almost surely) and our statements, to ease
the notation, will refer to the normalized case of $r_F=1$ (since otherwise we may work with $Y_k:= r_F^k X_k$). We also work with real random 
coefficients since, as it becomes apparent from the arguments, the complex case can be derived be considering the real and imaginary parts.

For establishing the strong singularity of the random power series we study the anti-concentration phenomenon for the sequence of their random partial sums. 
This, in turn, is quantified in terms of the L\'evy concentration function. Recall the following: 

\begin{definition} [L\'evy concentration function]
	Let $\xi$ be a random variable and let $\lambda \geq 0$. The L\'evy concentration function of $\xi$ at level $\lambda$ is defined by
		\begin{align}
			Q(\xi, \lambda) := \sup_{v\in \mathbb R} \mathbb P(v \leq \xi \leq v + \lambda).
		\end{align}
\end{definition}

This measure of dispersion was introduced in the works of Doeblin \& L\'evy \cite{DL} and further developed by Kolmogorov \cite{Kolmo}, Rogozin \cite{Rog-1, Rog-2}, 
Esseen \cite{Ess}, among others, for the study of the spread of sums of independent random variables. In our approach we shall use the following quantitative
form due to Rogozin.

\begin{lemma} [Rogozin, 1961] \label{lem:Rogo}
	Let $\xi_1, \ldots, \xi_N$ be independent random variables and let $\lambda_1, \ldots, \lambda_N>0$ such that $Q(\xi_k, \lambda_k)<1$ for $k\leq N$. 
	Then, for any $L>0$ if we set $\beta_k = \min\{L, \lambda_k/2\}$ we obtain\footnote{The original (combinatorial) proof in \cite{Rog-1}, and the alternative (Fourier analytic) approach offered by Esseen in \cite{Ess}, performed under the restriction $2L\geq \max_k \lambda_k$. However, a careful optimization argument in \cite{Ess} yields the above unrestricted form introducing the $\beta_k$'s.}
		\begin{align}
			Q(S_N, L) \leq CL \left( \sum_{k=1}^N \beta_k^2 (1-Q(\xi_k, \lambda_k)) \right)^{-1/2},
		\end{align}
	where $C>0$ is a universal constant and $S_N = \xi_1+\ldots + \xi_N$.
\end{lemma}

This result will be essential to derive anti-concentration results for averages of polynomials along arcs.
The next lemma says that anti-concentration estimates are inherited to mixtures. 
This result can be viewed as the counterpart of \cite[Lemma 4.7]{DV} in the context of lower deviations.

\begin{lemma} \label{lem:mix-sb}
	Let $(X, {\mathcal A}, \mu)$ and $(Y, {\mathcal B}, \nu)$ be two probability spaces and 
	let $h: X\times Y \to [0,\infty)$ be a ${\mathcal A} \otimes {\mathcal B}$-measurable function. Then, for all $t >0$ we have
		\begin{align}
			\mu\left( x\in X : \int_Y h(x,y) \, d\nu(y) \leq t \right) \leq 2 \int_Y \mu (x : h(x,y) \leq 2t ) \, d\nu(y).
		\end{align}
\end{lemma}
%


\noindent {\it Proof.} Without loss of generality we may assume that $t =1$. Let $B=\{h \leq 2\}$ and 
notice that Tonelli-Fubini's theorem allows for the right-hand side to be rewritten as $\int_X \nu (B_x) \, d\mu(x)$. Hence, it suffices to show that 
	\begin{align}
		\mathbf 1_A(x) \leq 2 \nu(B_x), 
	\end{align}
for all $x\in X$, where $A:= \left\{ x\in X: \int_Y h(x,y) \, d\nu(y) \leq 1\right\}$. We argue by contradiction. If this is not the case, then there exists 
$x\in X$ such that 
	\begin{align}
		\int_Y h(x, y) \, d\nu(y) \leq 1 \quad \textrm{and} \quad \nu(B_x) < 1/2.
	\end{align}
On the other hand, by Markov's inequality, we get
	\begin{align}
		\nu(Y \setminus B_x) = \nu (y\in Y : h(x,y) > 2) \leq \frac{1}{2} \int_Y h(x,y) \, d\nu(y) \leq \frac{1}{2},
	\end{align}
a contradiction. \prend

We will also need the Berry-Esseen estimate for the rate of convergence in the Central Limit Theorem which we include here as an auxiliary result.

\begin{lemma} [Berry 1941, Esseen 1942] \label{lem:BE}
	Let $Y=(Y_1, \ldots, Y_N)$ be a random vector on $\mathbb R^N$ with independent coordinates such that $\mathbb EY_k=0$, $\mathbb EY_k^2=1$, and 
	$\mathbb E|Y_k|^3<\infty$ for all $k\leq N$. Then, for any $\theta \in S^{N-1}=\{ v \in \mathbb R^N : \| v \|_2=1\}$ we have
		\begin{align}
			\sup_{u\in \mathbb R} |P(Z\leq u) - P(g \leq u) |  \leq C  \sum_{k=1}^N |\theta_k|^3 \mathbb E|Y_k|^3,
		\end{align}
	where $Z: =\langle Y, \theta \rangle$ and $g\sim N(0,1)$.
\end{lemma}

Now we turn to describe the probabilistic structure of the notion of $\psi$-(SNB) for random analytic functions. 
In what follows the reader interested in applications may think of the test function $\psi$ as $\psi(s) = s^p$ or $\psi(s) = \log^+ s$.

For clarity let us mention that for any fixed arc $I \subset (0, 2 \pi)$, $0<r<1$, and test function $\psi$, the mapping 
	\begin{align}
		\omega \mapsto Y_{I,r}(\omega): = \fint_I \psi(|F(re^{i\theta})|) \, d\theta \equiv \frac{1}{|I|}\int_I \psi(|F(re^{i\theta})|) \, d\theta,
	\end{align}
is a random variable viewed as a mixture of the r.v.s $\{\psi(|F(re^{i\theta})|)\}_{\theta\in I}$, 
where $I$ is considered as probability space equipped with the normalized Lebesgue measure. Next, the supremum of the latter variables over 
a dense set of radii, say
	\begin{align}
		Z_I:=\sup_{0 <q<1 \atop q\in \mathbb Q} Y_{I,q},
	\end{align}
defines a random process indexed by the sub-intervals of $(0,2\pi)$. Of course, we get
	\begin{align}
		Z_I \leq \sup_{0<r<1} \fint_I \psi(|F(re^{i\theta})|) \, d\theta,
	\end{align}
which shows that $Z_I$ can be used instead, in probabilistic considerations in proving (SNB). If, additionally, $\psi$ is continuous then we obtain equality
in the above comparison.

Having thus clarified measurability issues the next result informs us that the notion of strong natural boundary, for random power series, is a tail event.

\begin{fact} \label{fct:2-1}
	Let $\{X_k\}_{k=0}^\infty$ be a sequence of r.v.s and assume that the random power series 
	$F(\omega ; z) = \sum_{k=0}^\infty X_k(\omega) z^k$ has a.s. $r_F =1$. Let $\psi$ be a test function which has the 
	sub-additivity property \eqref{eq:sub-ad}.
	For any arc $I\subset (0,2\pi)$ we consider the r.v. $Z_I(\omega) := \sup_{0<q<1, q\in \mathbb Q} \fint_I \psi(|F(\omega ; qe^{i\theta})|) \, d\theta$. Then, the event
		\begin{align}
			G_I := \{ \omega : Z_I(\omega) < \infty\}
		\end{align}
	is in the tail $\sigma$-field ${\mathcal T} = \bigcap_{n=1}^\infty \sigma(\{X_k : k\geq n\})$. 
\end{fact}

\noindent {\it Proof of Fact \ref{fct:2-1}.} For $N=0,1,2,\ldots$, let $F_N (z) = \sum_{k=0}^N X_k z^k$, and let $R_N := F - F_N$. Note that for any $0<q<1$, 
by the triangle inequality we have
	\begin{align} \label{eq:tri-tail}
		\left| |R_N (qe^{i\theta})| - |F(qe^{i\theta})| \right| \leq |F_N(qe^{i\theta}) | \leq \sum_{k=0}^N |X_k|, \quad \textrm{a.s.}
	\end{align}
Therefore, we get from \eqref{eq:sub-ad} that
	\begin{align}
		Z_I /C- 1 -\psi\left( \sum_{k=0}^N|X_k| \right) \leq  \sup_q \fint_I \psi(|R_N (qe^{i\theta}) |) \, d\theta \leq C\left [ 1 +  Z_I + \psi \left(\sum_{k=0}^N |X_k| \right) \right],
	\end{align}
where we have also used \eqref{eq:tri-tail} and the monotonicity of $\psi$.	
It follows that
	\begin{align}
		G_I = \left\{ \omega : \sup_q \int_I \psi(|R_N(\omega ; qe^{i\theta}) |) \, d\theta <\infty \right\},
	\end{align}
and the latter event belongs to $\sigma (\{X_{N+1}, X_{N+2}, \ldots\})$. Since $N$ was arbitrary, the claim follows. \prend

\medskip

We end this section by including a standard argument showing that the event of $\psi$-(SNB) for random power 
series with independent coefficients obeys a zero-one law.

\begin{proposition} \label{prop:snb-0-1}
	Let $\{X_k\}_{k=0}^\infty$ be a sequence of independent r.v.s. and assume that the random power series $F(z) = \sum_{k=0}^\infty X_k z^k$ has $r_F=1$ a.s.
	Then, for any test function $\psi$ which satisfies \eqref{eq:sub-ad} we have the following dichotomy:
	Either $F$ has $|z|=1$ as $\psi$-(SNB) a.s. or there exists $I \subset (0,2\pi)$ so that 
			\begin{align}
				\sup_q \int_I \psi(|F(qe^{i\theta})| ) \, d\theta <\infty, \quad a.s.
			\end{align}
	\end{proposition}

\noindent {\it Proof.} Let ${\mathcal P} = \bigcup_{m=1}^\infty {\mathcal P}_m$ be the collection of all sub-intervals of the partitions ${\mathcal P}_m$ 
of $[0,2\pi)$ with ${\rm mesh}({\mathcal P}_m)= 2\pi/m$, i.e., 
	\begin{align}
		{\mathcal P}_m = \left\{J_m^k = \left[ \frac{2\pi}{m}(k-1), \frac{2\pi}{m} k \right) : k=1,2,\ldots,m \right\}, \quad m=1,2,\ldots .
	\end{align}
Note that for any ordinary $f(z)= \sum_{k=0}^\infty c_k z^k$ with $r_f=1$ the circle $|z|=1$ is {\it not} a $\psi$-(SNB) iff 
there exists $I\subset (0,2\pi)$ so that $\sup_r \int_I\psi( | f(re^{i\theta})| ) \, d\theta < \infty$. Equivalently, iff there exists $J_m^k$ which makes the 
latter quantity finite. Going back to the probabilistic framework, if we consider the random events
	\begin{align}
	G_{m,k} = \left\{ \omega \in \Omega : \sup_q \int_{J_m^k} \psi (|F(\omega ; qe^{i\theta})|) \, d\theta < \infty \right\}, \quad 1\leq k\leq m, \quad m=1,2,\ldots,
	\end{align}
and $G$ denotes the event that $F$ does not have $|z|=1$ as $\psi$-(SNB), then 
	\begin{align}
		G = \bigcup_{m=1}^\infty \bigcup_{k=1}^m G_{m,k}.
	\end{align}
Hence, if we assume that $|z|=1$ is not a $\psi$-(SNB) a.s., then $P(G)>0$ and by Boole's inequality we infer that there exist $m\in \mathbb N$ and 
$1\leq k\leq m$ so that $P(G_{m,k})>0$. Because of Fact \ref{fct:2-1} we know that $G_{m,k} \equiv G_{J_m^k}$ is in the tail $\sigma$-field, and since
$\{X_k\}$ are assumed independent, Kolmogorov's 0-1 Law implies that $P(G_{m,k})=1$, as claimed. \prend

\section{Main Results}


\subsection{Asymptotic anti-concentration} \label{s:3-1}

The following proposition motivates the anti-concentration conditions that will show up in the main result of this section, Theorem \ref{thm:s-nat-bd}.
It slightly generalizes a previous result from \cite{DV}.

\begin{proposition} \label{prop:nat-bd}
	Let $\{X_n\}_{n=0}^\infty$ be  a sequence of independent random variables. Suppose that there exists $\varepsilon>0$ with 
	\begin{align} \label{eq:cond-nb}
		\sum_{n=0}^\infty \left(1- Q(X_n,\varepsilon) \right) = \infty.
	\end{align} 
	Then, the random power series $F(z)=\sum_{n=0}^\infty X_n z^n $ has a.s. radius of convergence $r_F\leq 1$.
Furthermore, if\footnote{E.g., this can be achieved under the additional assumption that $\limsup_n \mathbb E|X_n|<\infty$.} 
$r_F=1$ a.s., then $F$ has $|z|=1$ as natural boundary a.s.
\end{proposition}

\noindent {\it Proof.} Let $b_n: = 1 - Q(X_n,\varepsilon)$. In particular, $P(|X_n|>\varepsilon / 2) \geq b_n$. The 2nd Borel-Cantelli Lemma along with \eqref{eq:cond-nb}
implies that $\{ |X_n| > \varepsilon/2 \; {\rm i.o.}\}$ is a sure event, hence $\limsup |X_n|^{1/n} \geq 1$ a.s.

For the natural boundary we may argue with Theorem \ref{thm:R-N}. If $|z|=1$ is not a 
natural boundary for $F(z) = \sum_{n=0}^\infty X_n z^n$, then there exists (deterministic) $(c_n)$ such that the $H(z) = F(z) - \sum_{n=0}^\infty c_n z^n$ 
has a.s. radius of convergence $r_H>1$. 
Then, we have that $\sum_{n=0}^\infty |X_n-c_n|$ converges a.s. This, in turn, implies (by the 2nd Borel-Cantelli lemma) that for every $\delta>0$ we have
	\begin{align}
		\sum_{n=0}^\infty \left( 1 - Q(X_n,\delta) \right) \leq \sum_{n=0}^\infty P(|X_n-c_n| > \delta/2) < \infty,
	\end{align}
which clearly contradicts the anti-concentration assumption \eqref{eq:cond-nb}.  \prend

\medskip



			

Although Proposition \ref{prop:nat-bd} establishes mere naturally boundary, it will be instructive for what follows (Theorem \ref{thm:s-nat-bd}) to compare
the conditions of Proposition \ref{prop:nat-bd} with those of Theorem \ref{thm:BS} of Breuer and Simon for strong natural boundary. The comparison
is in the following:

\begin{fact} \label{fct:dv-bs}
	Let $\xi$ be a random variable with ${\rm Var}[\xi] \geq \theta  >0$ and  $\mathbb E|\xi |^4 <\infty$. Then, we may conclude that
		\begin{align}
			1-Q\left( \xi,\sqrt{ 2\theta} \right) > c\frac{({\rm Var}[\xi] )^2}{\|\xi - \mathbb E\xi\|_4^4}.
		\end{align}
\end{fact}

\noindent {\it Proof of Fact \ref{fct:dv-bs}.} Using the standard ${\rm Var}[\xi] \leq \mathbb E|\xi-v|^2$ for any $v\in \mathbb R$, and the 
Paley-Zygmund inequality \cite{Stee} we obtain
	\begin{align} \label{eq:3-4}
		P( |\xi-v| > \sqrt{\theta/2}) \geq P \left (  |\xi -v|^2 > \frac{1}{2} \mathbb E|\xi-v|^2 \right) \geq \frac{1}{4} \frac{(\mathbb E|\xi -v|^2)^2}{\mathbb E|\xi-v|^4}.
	\end{align}
	We estimate the $\inf_{v\in \mathbb R} \frac{(\mathbb E|\xi -v|^2)^2}{\mathbb E|\xi-v|^4}$ from below. To this end, notice that
	\begin{align}
		(\mathbb E|\xi -v|^2)^2 = \left( {\rm Var}[\xi] + (v- \mathbb E\xi)^2\right)^2 \geq \left( {\rm Var}[\xi] \right)^2	+ (v- \mathbb E\xi)^4,
	\end{align}
and by Minkowski's inequality
	\begin{align}
		\mathbb E|\xi-v|^4 \leq \left( \| \xi - \mathbb E\xi\|_4 + |v - \mathbb E\xi| \right)^4 \leq 2^4 (\|\xi -\mathbb E\xi\|_4^4 + (v- \mathbb E\xi)^4).	
	\end{align}
It follows that 
	\begin{align} \label{eq:3-7}
		\inf_{v\in \mathbb R} \frac{(\mathbb E|\xi -v|^2)^2}{\mathbb E|\xi-v|^4} 
			\geq \inf_{v\in \mathbb R} \frac{\left( {\rm Var}[\xi] \right)^2	+ (v- \mathbb E\xi)^4}{2^4 (\|\xi -\mathbb E\xi\|_4^4 + (v- \mathbb E\xi)^4)} 
				= 2^{-4} \inf_{w\geq 0} \frac{\left( {\rm Var}[\xi] \right)^2 + w}{\|\xi -\mathbb E\xi\|_4^4 + w} = 2^{-4} \frac{ \left( {\rm Var}[\xi]\right)^2 }{ \|\xi - \mathbb E\xi \|_4^4}.
	\end{align}
Combining \eqref{eq:3-4} and \eqref{eq:3-7} the result follows with $c=2^{-6}$. \prend

\medskip

Continuing our discussion, taking into account Fact \ref{fct:dv-bs}, 
the assumptions in Theorem \ref{thm:BS}, and that $\mathbb E|X_n-\mathbb E[X_n]|^4 \leq (2M)^2 {\rm Var}[X_n]$ we infer
	\begin{align}
		1 - Q \left( X_n, \sqrt{\theta} \right) \geq c \frac{\theta}{M^2},
	\end{align}
for all $n$ with ${\rm Var}[X_n] > \theta/2$. On the other hand, we have the following simple estimate:
\begin{align} \label{eq:f-2}
		{\rm Var}[\xi] \geq \varepsilon^2 P( |\xi - \mathbb E\xi| > \varepsilon) \geq \varepsilon^2 (1 - Q(\xi , 2\varepsilon)), \quad \forall \, \varepsilon>0.
\end{align}
Therefore, under the assumptions of Theorem \ref{thm:BS}, Fact \ref{fct:dv-bs} and \eqref{eq:f-2} imply that  
	\begin{align} \label{from_var_to_wac}
		\limsup_n {\rm Var} [X_n] > 0 \quad \Longleftrightarrow \quad \limsup_n \left [ 1-Q(X_n,\varepsilon) \right ] >0 \quad
			 \Longrightarrow \quad \sum_{n =0}^\infty (1- Q(X_n,\varepsilon)) =\infty,
	\end{align}
for any $0< \varepsilon <\sqrt{\limsup_n {\rm Var}[X_n] }$.
Notice, however, that condition \eqref{eq:cond-nb} in conjunction with \eqref{eq:f-2} implies the divergence of the series of variances
which is a priori weaker than $\limsup_n {\rm Var}[X_n] >0$.
We elaborate further on this observation in Section \ref{S:bd-case}.

\medskip

The purpose of this section is to prove that under the assumption \eqref{eq:cond-nb} one can indeed infer that the random power series has (SNB), 
thereby extending the result of Breuer and Simon since no moment assumption is made. 
We will establish a stronger phenomenon under a technically weaker assumption. This technical condition, which can be viewed as weighted version
of \eqref{eq:cond-nb}, permits us to treat a greater range of probabilistic constructions which exhibit strong singularity. 
We comment on its utility in Remark \ref{rem:main-snb-2}. Our result reads as follows:

\begin{theorem} \label{thm:s-nat-bd}
Let $\{X_k\}_{k=0}^\infty$ be a sequence of independent random variables on a probability space $(\Omega, \Sigma, P)$
and assume that the random power series $F(\omega; z) = \sum_{k=0}^\infty X_k(\omega) z^k$ has a.s. radius of convergence $r_F=1$. 
Suppose that there exists a bounded sequence $(t_k)$, $t_k\geq 0$ such that 
	\begin{align} \label{eq:cond-nb-2}
		\sum_{k=0}^\infty t_k^2 \left( 1 - Q(X_k, t_k) \right) = \infty.
	\end{align} 
Then, for any test function $\psi$ we have that the random power series
$F$ a.s. has $|z|=1$ as $\psi$-(SNB). 
In particular, for a.e. $\omega$ the random power series $F(\omega ; \cdot)$ has the circle $\{z : |z|=1\}$ as  Nevanlinna-(SNB).
\end{theorem}

The key ingredients of the proof are: (i) Rogozin's inequality (Lemma \ref{lem:Rogo}), and (ii) the small-ball estimate under mixtures (Lemma \ref{lem:mix-sb})
described in Subsection \ref{S:2-2}.

\bigskip

Now we turn to proving the main result. First we introduce some notation for typographical convenience.

\bigskip

\noindent {\bf Notation.} We fix $t_k \geq 0$ for which \eqref{eq:cond-nb-2} holds true. Let $q_k : = Q(X_k, t_k)$. Then, 
\eqref{eq:cond-nb-2} becomes $\sum_{k=0}^\infty t_k^2(1-q_k) = \infty$. We consider the auxiliary function $A(r)$ for $0<r<1$ defined by
	\begin{align}
		A(r) : = \sum_{k=0}^\infty t_k^2(1-q_k) r^{2k}.
	\end{align}
Note that $A(r)<\infty$ for all $0<r<1$ and $\lim_{r\uparrow 1} A(r) = \infty$. Let also $F_N(\omega ; z)$ and 
$A_N(r)$ be the corresponding partial sums for $F(\omega ; z)$ and $A(r)$ respectively. For each interval $I\subset (0,2\pi)$ and $0<r<1$ we define
the random variable 
	\begin{align}
		Y_{I,r} (\omega) : = \fint_I \psi(|F(\omega ; re^{i \theta})| \, d\theta \equiv \frac{1}{|I|} \int_I \psi(|F(\omega ; re^{i\theta})|) \, d\theta.
	\end{align}
Last, we define the countable set
	\begin{align}
		\Theta = \bigcup_{n=1}^\infty \left \{ \theta \in (0,2 \pi) : \sin (2n\theta) =0 \right\}.
	\end{align}

\medskip

\noindent {\it Proof of Theorem \ref{thm:s-nat-bd}.} With the above notation we have the following:

\begin{claim} \label{claim:main-sb}
	Let $I \subset (0, 2\pi)$ be an interval and let $0<r<1$. Then, for all $t\geq  \sup_k t_k $ we have
		\begin{align} \label{eq:main-sb}
			P \left( Y_{I,r} \leq t \right) \leq \frac{Ct}{ \sqrt{A(r)} },
		\end{align}
	where $C>0$ is a universal constant.
\end{claim}

\noindent {\it Proof of Claim \ref{claim:main-sb}.} Let $N\geq 1$ be sufficiently large so that 
$\sum_{k=0}^N t_k^2(1-q_k) >1$, and let $\theta \in I^\ast := I \setminus \Theta$ be arbitrary, but fixed. Note that, either
$\sum_{k=0}^N t_k^2(1-q_k )r^{2k} \cos^2 k\theta \geq \frac{1}{2}A_N(r)$ or $\sum_{k=0}^N t_k^2(1-q_k )r^{2k} \sin^2 k\theta \geq \frac{1}{2}A_N(r)$. 
We assume that the former case holds (we work similarly in the latter case). Then, we may write
	\begin{align}
		P \left( |F_N(\omega ; re^{i\theta}) |\leq t \right) & \leq P \left( \left| \sum_{k=0}^N X_k(\omega) r^k \cos k\theta \right| \leq t \right) \leq 
		\frac{C t}{\sqrt{ \sum_{k=0}^N (t_k r^k \cos k\theta)^2 (1-q_k) } },
	\end{align}
where in the last passage we have applied Rogozin's inequality 
for the independent random variables $\xi_k = X_k r^k \cos k \theta$, 
for $\lambda_k = t_k r^k |\cos k\theta|$, and for $L = t$ which clearly satisfies the required restriction 
``$L \geq  \frac{1}{2} \max_k \lambda_k$''. Because we have assumed the former case, the latter estimate becomes
	\begin{align} \label{eq:main-sb-1}
		P \left( |F_N(\omega; re^{i\theta})| \leq t \right) \leq \frac{\sqrt{2}C t}{ \sqrt{A_N(r)}},
	\end{align}
for any $N$ and $\theta$ as above. Since $N$ was arbitrarily large, Fatou's lemma and the a.s. (uniform) convergence of $F_N$ to $F$ yield that 
	\begin{align} \label{eq:main-sb-2}
		P \left(  |F(\omega ; re^{i\theta})| < t  \right) \leq \liminf_N P\left( |F_N(\omega ; re^{i\theta})| \leq t\right ) \stackrel {\eqref{eq:main-sb-1}} \leq \frac{C' t}{\sqrt{A(r)}}.
	\end{align}
The monotonicity of $\psi$ yields that 
	\begin{align}
			P\left(  \psi (|F(\omega ; re^{i\theta})| ) < \psi (t)  \right) \leq \frac{C't}{ \sqrt{A(r)}},
		\end{align}
	On the other hand, since $\theta \in I^\ast$ was arbitrary, an appeal to Lemma \ref{lem:mix-sb} for ``$(Y,\nu)$ being the $I^\ast$ 
equipped with the normalized Lebesgue measure'' and ``$\alpha= \psi(t)$'' yields 
	\begin{align} \label{eq:main-sb-3}
		P \left( \fint_{I^\ast} \psi(|F(\omega ; re^{i\theta}) |) \, d\theta < \psi(t) \right) \leq \frac{4 C' t}{\sqrt{A(r)} }.
	\end{align}
It remains to notice that $Y_{I,r} (\omega) \stackrel{\rm a.s.} =  \fint_{I^\ast} \psi( |F(\omega ; re^{i\theta}) | ) \, d\theta$, 
since $\Theta$ is a null set. This proves the Claim. \prend

\medskip

Continuing with the proof of Theorem \ref{thm:s-nat-bd}, 
if we employ \eqref{eq:main-sb} we may conclude the assertion as follows:
Fix $I\subset (0,2\pi)$. Let $r_k\in (0,1)$ such 
that\footnote{The mapping $r \mapsto A(r)$ is increasing, continuous with $\lim_{r\downarrow 0} A(r) =0$ and $\lim_{r\uparrow 1}A(r)=\infty$.} 
$A(r_k)=k^6$. Applying \eqref{eq:main-sb} for $r=r_k$ and $t=k$ we obtain for the r.v.s $Z_{I,k} := Y_{I,r_k}$ and the events
$U_{I,k} := \{Z_{I,k} \leq \psi(k)\}$ the following:
	\begin{align}
		P \left( Z_{I,k} \leq \psi(k) \right) \leq \frac{C}{ k^2} \quad \Longrightarrow \quad P(\limsup_k U_{I,k}) =0.
	\end{align}
Finally, if $U_{I,\infty} := \limsup_k U_{I,k}$, then $P(U_{I, \infty}^c)=1$ for every $I\subset (0,2\pi)$ and if $\omega \notin U_{I,\infty}$, then there exists 
$k_0(\omega)\in \mathbb N$ so that $Z_{I,k} > k$ for all $k\geq k_0$. The result readily follows if we consider a countable base of intervals 
for the topology of $(0,2\pi)$. \prend

\begin{remark}
	Let us point out that the boundedness on the anti-concentration levels $(t_k)$ is not a 
	mere technicality of the approach followed. The following
	example shows that we may drop the boundedness of $(t_k)$ and still have the asymptotic anti-concentration condition \eqref{eq:cond-nb-2} 
	whereas the random series is a.s. extendable beyond the unit circle.
\end{remark}

\begin{example} \label{ex:nec}
	Let $\{Y_k\}_{k=0}^\infty$ be independent r.v.'s with $P(Y_k =0) = (k+1)^{-2}$ and $P(Y_k =k+1) =1-(k+1)^{-2}$ for $k=0,1,\ldots$.
	For each $k\geq 1$, we have
		\begin{align}
			{\rm Var}[Y_k]= 1-\frac{1}{(k+1)^2}, \quad \textrm{and} \quad Q(Y_k, \delta) = \begin{cases}
							1-\frac{1}{(k+1)^2}, & 0< \delta<k+1 \\
							1, & \textrm{otherwise} 
							\end{cases}.
		\end{align}
	Therefore, if for some sequence $(\delta_k)$ of positive numbers we have
		\begin{align}
			\sum_{k=1}^\infty \delta_k^2 (1-Q(Y_k, \delta_k))=  \sum_{k=1}^\infty \frac{\delta^2_k}{(k+1)^2} \mathbf 1_{\{\delta_k < k+1\}} =\infty,
		\end{align}
	we infer that $\limsup (k^{-1/4} \delta_k ) = \infty$. In particular, $(\delta_k)$ cannot be bounded.
On the other hand, for the random power series $F(\omega ; z) := \sum_{k=0}^\infty Y_k(\omega) z^k$ we have the following:
		\begin{itemize}
			\item $r_F=1$ a.s.
			\item For almost every $\omega \in \Omega$ the realization $F(\omega; z)$ satisfies that
				\begin{align}
					\mathbb D \ni z\mapsto F(\omega ; z) - \frac{1}{(1-z)^2}
				\end{align}
			is a polynomial. In particular, a.s. the unit circle is {\it not} a natural boundary for $F$.
		\end{itemize}
		
		Indeed; by the 1st Borel-Cantelli Lemma we obtain that the events $E_k=\{Y_k=k+1\}$ satisfy $P(\limsup E_k^c)=0$. 
		Thus, for every $\omega \notin \limsup E_k^c$ there 
		exists $k_0 = k_0(\omega) \in \mathbb N$ such that $\omega \in E_k$ for all $k\geq k_0$. Hence, we find
		that $\limsup |Y_k(\omega)|^{1/k} = \limsup (k+1)^{1/k}=1$; this proves that $r_F=1$ a.s. Moreover, we may write
			\begin{align*}
				F(\omega ; z) - \sum_{k=0}^{k_0} Y_k(\omega) z^k &= \sum_{k=k_0+1}^\infty Y_k(\omega) z^k \\
											&= \sum_{k=0}^\infty (k+1)z^k - \sum_{k=1}^{k_0} (k+1)z^k \\
											& = \left( \frac{1}{1-z}\right)' - \sum_{k=1}^{k_0} (k+1)z^k,
			\end{align*}
		which proves that $F(\omega ; z) - \frac{1}{(1-z)^2}$ is a polynomial as claimed. \prend
		
\end{example}

\begin{remark} \label{rem:main-snb-2}
	We have already explained how condition \eqref{eq:cond-nb} (and hence, the condition \eqref{eq:cond-nb-2}) 
	extends the assumptions of Theorem \ref{thm:BS}. The next example
	illustrates the utility of \eqref{eq:cond-nb-2}, which can be viewed as a ``weighted'' version of \eqref{eq:cond-nb}. 
	
	E.g., consider an independent Poisson trial $(\delta_k)$ (a.k.a. independent Bernoulli r.v.s $\delta_k$ with mean $p_k\in (0,1)$), 
	and let $(c_k)$ be a bounded sequence of non-zero numbers such that $c_k^2\min\{p_k,1-p_k\} \to 0$ and $\sum c_k^2 \min\{p_k,1-p_k\} = \infty$. 
	Then, the random power series $\sum_k c_k \delta_k  z^k$ 
	has almost surely $\psi$-(SNB) at $\{z: |z|=1 \}$. Indeed; for the r.v.s ``$X_k = c_k \delta_k$'' and for ``$t_k = |c_k|/2$'' in Theorem \ref{thm:s-nat-bd}, we have
	``$Q(X_k ,t_k) = \max\{p_k,1-p_k\}$'' and thus 
		\[
			\sum_{k=0}^\infty t_k^2 \left( 1 - Q(X_k, t_k) \right) = \frac{1}{4} \sum_{k=0}^\infty c_k^2 \min\{p_k,1-p_k\} = \infty.
		\]
	On the other hand, notice that $|X_k| = |c_k| \leq M:= \sup_k |c_k| < \infty$ a.s. and ${\rm Var} [X_k] = c_k^2 p_k(1-p_k) \to 0$. 
\end{remark}

We conclude this section with an easy consequence about random power series with i.i.d. coefficients. 

\begin{corollary}
	Let $\{X_k\}_{k=0}^\infty$ be i.i.d. r.v.s. with non-degenerate distribution, and $\mathbb E [ \log^+|X_1| ]<\infty$. 
	Then, the random power series $\sum_{k=0}^\infty X_k z^k$ has a.s. $\psi$-(SNB) at $|z|=1$ for all test functions $\psi$.
\end{corollary}

\noindent {\it Proof.} Since $X_1$ has non-degenerate distribution we get $P(|X_1|>0)>0$. This implies the equivalence 
	\[
		\mathbb E [ \log^+ |X_1|]<\infty \quad \Longleftrightarrow \quad r_F=1 \; \textrm{a.s.},
	\]
see e.g., \cite[Exercise 22.10]{Bil}. Since $P(|X_1|>0)>0$, there exists $\delta>0$ so that $p=P(|X_1| > \delta/2) >0$. It follows that $1-Q(X_k, \delta) \geq p$ 
for all $k$, hence \eqref{eq:cond-nb-2} is fulfilled with $t_k=\delta$. \prend

\subsection{The symmetric case} \label{SymmetricSection}

In this section we will use our main result (Theorem \ref{thm:s-nat-bd}) to settle the case of independent symmetric coefficients. 
The argument we follow is well known in the convergence of random series (see e.g., \cite[Th\'eor\`eme 7]{Mar-Zyg}) and rests on the probabilistic fact 
that if $\{X_n\}$ is a sequence of independent, symmetric r.v.s, and $\{\varepsilon_n\}$ a sequence of independent Rademacher
r.v.s, which are independent of $\{X_n\}$, then $\{X_n\}$ and $\{\varepsilon_n X_n\}$ are equidistributed. This permits us, after conditioning on 
$\{X_n\}$, to establish the desired property for a Rademacher sequence. This technique is customary referred to as the {\it reduction principle}, see \cite[p.8-9]{Kah}.

Before we study the case of Rademacher power series we mention a result of independent interest: We observe that
random power series with independent symmetric coefficients enjoy a spreading property that turns local boundary integrability to global.
The argument uses the decomposition scheme for random power series into ones with rotation symmetry, 
as in the proof of standard natural boundary (see \cite{Kah}.)

\begin{proposition}
	Let $\{X_k\}_{k=0}^\infty$ be independent and symmetric r.v.s and let $\psi$ be a test function which satisfies
	the sub-additivity condition
		\begin{align}
			\psi(t+s) \leq K[1 + \psi(t) + \psi(s)], \quad t,s \geq 0, \quad (K>0).
		\end{align} 
	Then, the random power series $F(\omega ; z) = \sum_{k=0}^\infty X_k(\omega) z^k$
	either a.s. belongs to $H^\psi$ or a.s. has $\psi$-(SNB).
\end{proposition}

\noindent {\it Proof.} For concreteness we prove the probabilistic dichotomy for the $N$ space, i.e., for $\psi(t) = \log^+t$; 
the general case is treated similarly. 
If $F$ does {\it not} have $N$-natural boundary a.s., then Proposition \ref{prop:snb-0-1} yields the existence of an arc $I\subset (0,2\pi)$ so that for almost
every $\omega$ we have
	\begin{align} \label{eq:N-sym}			
		M(\omega):=\sup_{r} \int_I \log^+ |F(\omega ; re^{i \theta})| \, d\theta <\infty.
	\end{align}
Let $\ell\in \mathbb N$ so that $2\pi /\ell < |I|/2$. Then, there exists $s\in \{1,\ldots,\ell\}$ so that $J_s= [\frac{2\pi}{\ell}(s-1), \frac{2\pi}{\ell} s) \subset I$.
For $k=0,1,\ldots, \ell-1$ we introduce the sequences of signs $(\varepsilon_{k,j})_{j=0}^\infty$ defined by
	\begin{align}
		\varepsilon_{k,j} = \begin{cases}
							1, & j\nequiv k \mod \ell \\
							-1, & j \equiv k \mod \ell
						\end{cases}.
	\end{align}
Let $F_k (z) = \sum_{j=0}^\infty  \varepsilon_{k,j}X_j z^j$. Because of the symmetry of $X_k$'s the random power series $F$ and $F_k$ are equi-measurable,
hence for almost every $\omega$ the function $F_k$ satisfies \eqref{eq:N-sym}, too. It follows that for almost every $\omega$ we have
	\begin{align}	
		\sup_r \int_I \log^+|F(\omega ; r e^{i\theta}) - F_k(\omega ; re^{i\theta})| \, d\theta \leq |I| \log 2 + 2M(\omega) < \infty,
	\end{align}
where we have used \eqref{fct:tria-log+} for $m=2$. On the other hand, we have 
	\begin{align} \label{eq:N-sym-1}
		F(z) - F_k(z) = 2\sum_{j=0}^\infty X_{j\ell+k}z^{j\ell+k} = 2z^kH_k(z), \quad H_k(z):= \sum_{j=0}^\infty X_{j \ell+k} z^{j \ell}.
	\end{align}
For the $H_k$'s we have the following properties: 
	\begin{itemize}
		\item Each $H_k$ is invariant under rotation by angle $2\pi/\ell$, i.e. $H_k(ze^{2\pi i/\ell})=H_k(z)$, $z\in \mathbb C$. 
		\item Because of the rotation invariance we derive
			\begin{align} \label{eq:N-sym-2}
				\int_{J_s} \log ^+ |r^kH_k( re^{i\theta})| \, d\theta = \int_{J_t} \log ^+ |r^kH_k( re^{i\theta})| \, d\theta,	
			\end{align}
			for $r>0$, and any $t,s \in \{1,\ldots, \ell\}$, where $J_s = \left[ \frac{2\pi}{\ell}(s-1), \frac{2\pi}{\ell}s \right)$.
		\item The $H_k$'s form a decomposition of $F$ as follows 
			\begin{align}
				F(z) = \sum_{k=0}^{\ell-1} z^k H_k(z).
			\end{align}
	\end{itemize}
Combining \eqref{eq:N-sym}, \eqref{eq:N-sym-1}, and $\eqref{eq:N-sym-2}$ we find for almost all $\omega$ that
	\begin{align}
		\sup_r \int_0^{2\pi} \log^+ |r^k H_k(\omega; r e^{i\theta})| \, d\theta \leq \ell \left( |I| \log 2+ 2M(\omega) \right) < \infty.
	\end{align}
Using the decomposition property of the $H_k$'s, the ``triangle inequality'' \eqref{fct:tria-log+} for $\log^+$, and the latter estimate 
we obtain for almost every $\omega$ that 
	\begin{align}\begin{split}
		\sup_r \int_0^{2\pi} \log^+ |F(\omega; r e^{i\theta})| \, d\theta & \leq 2\pi \log \ell + \sum_{k=0}^{\ell-1} \sup_r \int_ 0^{2\pi}\log^+ |r^k H_k(\omega; r e^{i\theta})| \, d\theta \\
					& \leq 2\pi \log \ell + \ell^2 \left( |I| \log 2 + 2M(\omega) \right) <\infty.
		\end{split}
	\end{align}
This proves that $F$ is in $N$-space a.s. \prend

\begin{proposition} [Rademacher power series] \label{prop:Rad-snb}
	Let $\{\varepsilon_k\}_{k=0}^\infty$ be independent Rademacher r.v.s and let $(c_k)$ be a deterministic sequence of numbers so that 
	$\limsup |c_k|^{1/k}=1$. Then, for the random power series $F(z) = \sum_{k=0}^\infty c_k \varepsilon_k z^k$ we have the following:
		\begin{itemize}
			\item If $\sum_{k=0}^\infty |c_k|^2<\infty$, then $F\in H^p(\mathbb D)$ for all $0<p<\infty$ a.s. 
			
			\item If $\sum_{k=0}^\infty |c_k|^2 = \infty$, then for any test function $\psi $, $F$ has a.s. the circle $|z|=1$ as $\psi$-(SNB).
		\end{itemize}
\end{proposition}

\noindent {\it Proof.} The assumption clearly implies that $r_F=1$ a.s. 
The first item is well known. Since the proof is short we sketch it for reader's convenience. For every fixed $1\leq p<\infty$ and $0<r<1$, 
Khintchine's inequality \cite[Exercise 2.6.5.]{Ver-book} implies
	\begin{align}
		\mathbb E \left [ \fint_0^{2\pi} |F(re^{i\theta})|^p \, d\theta \right] \leq B_p^p\left ( \sum_{k = 0}^\infty |c_k|^2 r^{2k} \right)^{p/2} \leq B_p^p \|c\|_2^{p}. 
	\end{align}
Note that $\|F\|_{H^p}^p:=\sup_{0<r<1} \fint_0^{2\pi} |F(re^{i\theta})|^p \, d\theta = \lim_{r\uparrow 1} \fint_0^{2\pi} |F(re^{i\theta})|^p \, d\theta$, thus Fatou's lemma yields
	\begin{align}
		\mathbb E [\|F\|_{H^p}^p] \leq \lim_{r\uparrow 1} \mathbb E\left[ \fint_0^{2\pi} |F(re^{i\theta})|^p \, d\theta \right ] \leq B_p^p \|c\|_2^{p}.
	\end{align}
The claim now easily follows if take into account the fact that $B_p \lesssim \sqrt{p}$ for $p\geq 1$ and the 1st Borel-Cantelli lemma. 

For the second assertion we employ Theorem \ref{thm:s-nat-bd} for ``$X_k= c_k \varepsilon_k$'' and ``$t_k = \min\{1, |c_k|\}$''. Indeed; 
$0\leq t_k\leq 1$, $t_k=0$ iff $c_k=0$, and 
	\begin{align*}
		\sum_{k=0}^\infty t_k^2 [1- Q(X_k, t_k)]= \frac{1}{2} \sum_{k=0}^\infty t_k^2.
	\end{align*}
To conclude notice that, since $\sum_{k=0}^\infty |c_k|^2 = \infty$, we also have\footnote{Let $J=\{j : |c_j|>1\}$. We distinguish two cases: If $J$ is infinite
then $\sum_k t_k^2 \geq \sum_{j\in J} t_j^2 = \sum_{j\in J} 1=\infty$. If $J$ is finite, then $\sum_k t_k^2\geq \sum_{j\notin J} t_j^2 = \sum_{j\notin J}|c_j|^2=\infty$.} 
$\sum_{k=0}^\infty t_k^2=\infty$. \prend

\begin{theorem} \label{thm:snb-sym}
	Let $\{X_k\}_{k=0}^\infty$ be a sequence of independent and symmetric random variables and assume that the random
	power series $F(z) = \sum_{k=0}^\infty X_k z^k$ has $r_F=1$ a.s. We have the following:
		\begin{itemize}
			\item If $\sum_{k=0}^\infty |X_k|^2 < \infty$ a.s., then $F\in H^p(\mathbb D)$ for all $0<p<\infty$ a.s.	
			
			\item If $\sum_{k=0}^\infty |X_k|^2 = \infty$ a.s., then for any test function $\psi$, $F$ has the circle $|z|=1$ as $\psi$-(SNB) a.s.
		\end{itemize}
\end{theorem}

\noindent{\it Proof.} We shall use independent Rademacher random variables $\{\varepsilon_k\}_{k=0}^\infty$ that are also independent of $\{X_k\}_{k=0}^\infty$. 
For this, we may take  $\{\varepsilon_k\}_{k=0}^\infty$ on $((0,1), {\cal B}, m)$ to be $\varepsilon_k(t) = {\rm sgn} (\sin (2^k\pi t))$ 
and then each $\varepsilon_k X_k$ to be on $((0,1) \times \Omega, m\times P)$.
Recall then that, as the $X_k$ are symmetric, the distribution
of $\{\varepsilon_k X_k\}_{k=0}^\infty$ is the same as that of $\{X_k\}_{k=0}^\infty$. A fortiori, the random functions $F(z) = \sum_{k=0}^\infty X_k z^k$ 
and $H(z)= \sum_{k = 0}^\infty \varepsilon_k X_k z^k$ are equimeasurable. 

We prove now the second conclusion of the Theorem (and work similarly for the first conclusion).
To this end, let $D= \{\omega\in \Omega : \sum_{k=0}^\infty |X_k(\omega)|^2 = \infty\}$, 
let $A$ be the event on $(\Omega, P)$ that $F$ has $|z|=1$ as $\psi$-(SNB), and let $B$ be the event on $((0,1) \times \Omega, m\times P)$ 
that $H$ has $|z|=1$ as $\psi$-(SNB). Then the equimeasurability of $F$ and $H$, Tonelli-Fubini's theorem, and the fact that $D$ is a sure event yield
	\begin{align}
		P(A) = (m\times P)(B) = \int_\Omega m(B^\omega) \, dP(\omega) = \int_D m(B^\omega) \, dP(\omega).
	\end{align}
Finally, note that for every fixed $\omega\in D$ the Rademacher random series $\sum_{k = 0}^\infty X_k(\omega) \varepsilon_k z^k$ satisfies the condition (ii) in 
Proposition \ref{prop:Rad-snb}. That is, $m(B^\omega)=1$. 

\prend

\subsection{Bounded random variables}  \label{S:bd-case}

The purpose of this section is to provide 
criteria for the presence of a strong natural boundary for $F(z) = \sum_{k=0}^\infty X_k z^k$ with $\{X_k\}_{k=0}^\infty$ uniformly bounded. 
For such a power series, we show that strong natural boundaries exist whenever the series of variances diverges by replacing the estimate \eqref{claim:main-sb}.
I.e.\ under the assumption of uniform boundedness we establish strong natural boundaries for a larger class than the class of power series satisfying 
the weighted anti-concentration condition \eqref{eq:cond-nb-2}. 
In particular, we revisit the variance condition in \ref{thm:BS} and show how both the assumption
and the conclusion of this theorem can be improved. 

For motivation, note the example in Remark \ref{rem:main-snb-2} which satisfies the weighted anti-concentration condition \eqref{eq:cond-nb-2}, and
uses uniformly bounded $X_k$'s with Var$[X_k]\to 0$. 
On the other hand, $X_k$'s satisfying the conditions from Theorem \ref{thm:BS} (uniform boundedness and Var$[X_k]\nrightarrow 0$) 
also satisfy the weighted anti-concentration condition thanks to \eqref{from_var_to_wac}. Therefore, 
the weighted anti-concentration condition \eqref{eq:cond-nb-2} is genuinely weaker than the conditions of Theorem \ref{thm:BS}. 

At the same time, condition \eqref{eq:cond-nb-2} implies, via \eqref{eq:f-2}, that $\sum_{k=0}^\infty {\rm Var}[X_k] =\infty$. 
Then when trying to clarify the existence of strong natural boundaries, and in view of our Theorem \ref{thm:s-nat-bd}, the following question suggests itself:	
	\begin{quote} 
	Is it true that if $\{\xi_n\}$ are random variables with $\mathbb E |\xi_n|^2<\infty$, the following equivalence
		\begin{align} \label{eq:var-anti}
			\sum_{n=1}^\infty {\rm Var}[\xi_n] = \infty \quad \stackrel{?}\Longleftrightarrow 
			\quad 
			\sum_{n=0}^\infty t_n^2 \left( 1 - Q(\xi_n, t_n) \right) \quad \textrm{for bounded $(t_n)$}
		\end{align}
	holds? What if $\{ \xi_n \}$ is uniformly bounded? 
	\end{quote}
	Example \ref{ex:nec} shows that the answer is clearly negative without assuming uniform boundedness. In the presence of uniform boundedness, one may hope
for a reversal of \eqref{eq:f-2}. The following (optimal) estimate shows to what extend this reversal is possible and, at the same time, suggests Example \ref{ex:vwa-opt} 
below that answers \label{eq:var-anti} in the negative even for uniformly bounded coefficients. 

\begin{fact} \label{fct:rev-wL2}
	Let $X$ be a r.v. in $L_\infty(\Omega, {\mathcal E}, P)$. Then, we have
		\begin{align} \label{eq:rev-wL2}
			{\rm Var}[X] \lesssim \log\left (\frac{3\|X\|_\infty }{\sqrt{{\rm Var}[X]}} \right) \cdot \sup_{\delta >0} \left\{ \delta^2[1-Q(X,\delta) ]\right\} .
		\end{align}
\end{fact}

We defer the proof of this result in the Appendix.

\begin{example} \label{ex:vwa-opt} 
	Let $\Omega =(-1/2,1/2)$ be equipped with the Borel $\sigma$-field and the standard Lebesgue measure as the underlying probability space. 
	We consider the mono-parametric family of r.v.s $\{X_\alpha\}$ on $\Omega$ defined by
		\begin{align}
			X_\alpha (\omega) = {\rm sgn}(\omega) \sqrt{\frac{\alpha}{|\omega|} } \mathbf 1_{(-\alpha, \alpha)^c}(\omega), \quad \omega \in \Omega, \quad (0<\alpha<1/2). 
		\end{align}
	The desired sequence $\{\xi_k\}$ will be built out of $\{X_\alpha\}$ for an appropriate choice of a sequence $\alpha=\alpha_k$. First, note that $X_\alpha$'s satisfy the following:
		\begin{enumerate}
			\item $\|X_\alpha\|_{L^\infty}=1$ for all $\alpha\in (0, 1/2)$.
			\item $X_\alpha$ is symmetric, and ${\rm Var}[X_\alpha]=\mathbb E[X_\alpha^2] = 2 \alpha \int_{\alpha}^{1/2} \frac{d\omega}{\omega} = -2\alpha \log(2\alpha)$.
			\item For $\delta>0$ we have
				\begin{align}
					P(|X_\alpha|>\delta)  = \begin{cases}
											1-2\alpha, & \delta \leq \sqrt{2\alpha} \\
											 2\alpha \left(\delta^{-2}-1\right), & \sqrt{2\alpha} < \delta <1 \\
											 0, & \delta\geq 1
										\end{cases}.
				\end{align}
			\item Therefore, we obtain
				\begin{align}
					\delta^2 [1-Q(X_\alpha, \delta)] \leq \delta^2 P(|X_\alpha| > \delta/2) \leq 8\alpha.
				\end{align}
		\end{enumerate}
Now we are ready to define the r.v.s $\{\xi_k\}$ refuting equivalence \eqref{eq:var-anti}: Let $(\alpha_k)$ 
be the sequence of positive numbers defined by $\alpha_k^{-1}= 3k[\log(ek)]^2$ 
for $k=1,2,\ldots$ and let $\xi_k:= X_{\alpha_k}$. Clearly, $\|\xi_k\|_{L^\infty}=1$, and
	\begin{align}
		{\rm Var}[\xi_k] = -2\alpha_k \log(2\alpha_k) \asymp \frac{1}{k[\log(ek )]}, \quad 
				\sup_{\delta>0} \{ \delta^2 [1-Q(\xi_k, \delta)]\} \leq 8\alpha_k \asymp \frac{1}{k[\log(ek)]^2},
	\end{align}
for $k=1,2,\ldots$.		
\end{example}

Given that equivalence \eqref{eq:var-anti} is false the following problem arise naturally:

\begin{quote} \label{Q:prob-2}
	Let $\{X_n\}_{n=0}^\infty$ be a sequence of independent random variables with $\mathbb E|X_n|^2<\infty$ for each $n$ and suppose that the random power series 
	$F(z) = \sum_{n=0}^\infty X_n z^n$ has radius of convergence $r_F=1$. Is it true that under the following condition
		\begin{align} \label{eq:cond-var}
			\sum_{n=0}^\infty {\rm Var}[X_n] =\infty,
		\end{align}
	the function $F$ has the unit circle as (SNB)? What if $\{X_n\}_{n=0}^\infty$ is uniformly bounded?
\end{quote}

Again, Example \ref{ex:nec} shows that the answer is negative if we do not assume uniform boundedness. We affirm this question when $\{X_n\}_{n=0}^\infty$ is uniformly
bounded:

\begin{theorem} \label{thm:snb-be}
	Let $\{X_n\}_{n=0}^\infty$ be a sequence of independent r.v.s with $|X_n| \leq M$ a.s. for all $n$ and $\sum_{n=0}^\infty {\rm Var}[X_n] = \infty$. 
	Then, the random power series $F(z) = \sum_{n=0}^\infty X_n z^n$ satisfies the following:
		\begin{itemize}
			\item $r_F=1$ a.s., and 
			\item for any test function $\psi$, $F$ has $|z|=1$ as $\psi$-(SNB) a.s.
		\end{itemize}
\end{theorem}

We shall infer the Theorem from the next more general result. 
The approach is in the spirit of the proof of Theorem \ref{thm:s-nat-bd}, and yields a new criterion by assuming a condition on third moments only.

\begin{proposition} \label{prop:3rd-var}
	Let $\{X_n\}_{n=0}^\infty$ be independent r.v.s. with $\sup_n \mathbb E |X_n-\mathbb E[X_n]|^3 / {\rm Var}[X_n] <\infty$, 
	and $\sum_{n=0}^\infty {\rm Var}[X_n] =\infty$. 
	Suppose that the random power series $F(z) = \sum_{n=0}^\infty X_n  z^n$ has $r_F=1$ a.s. Then, $F$ has $|z|=1$ as $\psi$-(SNB) a.s.
\end{proposition}

\noindent {\it Proof. } The main ingredient of the proof, as occurred in the proof of Theorem \ref{thm:s-nat-bd}, is the derivation of a distributional inequality akin
to \eqref{eq:main-sb-1}. Back then, it was achieved by leveraging Rogozin's inequality; this time we employ the Berry-Esseen estimate (Lemma \ref{lem:BE}). 
Let $K:= \sup_n \mathbb E |X_n-\mathbb E[X_n]|^3 / {\rm Var}[X_n]$. We have the following:

\begin{claim} \label{clm:snb-var}
	Let $(X_n)$ be as above. Fix $0<r<1$ and $I\subset (0,2\pi)$. Then, for all $t>0$ we have
		\begin{align}
			P \left( \fint_I \psi(|F(re^{i\theta})|) \, d\theta < \psi \left( c t \sqrt{V(r)} \right) \right) \leq C \left( t + \frac{K}{[V(r)]^{1/2}} \right),
		\end{align}
	where $V(r):= \sum_{k=0}^\infty r^{2k}v_k$ and $v_k ={\rm Var}[X_k]$.
\end{claim}

\smallskip

\noindent {\it Proof of Claim \ref{clm:snb-var}.} 
We begin with a consequence of Berry-Esseen estimate (Lemma \ref{lem:BE}) which is tailored to our needs: if $Y_1, \ldots, Y_n$ are independent and 
$S_n = Y_1+\ldots + Y_n$, then for all $t>0$ we have
	\begin{align} \label{eq:BE-2}
		P\left( |S_n| \leq t\sqrt{ {\rm Var}[S_n] }\right) \leq C \left(t + \frac{\sum_{j=1}^n \mathbb E|Y_j- \mathbb E[Y_j]|^3}{( {\rm Var}[S_n])^{3/2}} \right).
	\end{align}
	
\noindent 

We fix $r\in (0,1)$ and $\theta \in I$. We write $F_N(z) = \sum_{k=0}^N X_k z^k$. Then, at least one of the following holds true:	
	\begin{align}
		\sum_{k=0}^N r^{2k} v_k \cos^2(k\theta) \geq \frac{1}{2} V_N(r), \quad \sum_{k=0}^N r^{2k} v_k \sin^2(k\theta) \geq \frac{1}{2} V_N(r).	
	\end{align}
Let's assume the former case (we work similarly in the latter). Now we apply the BE estimate \eqref{eq:BE-2} for ``$Y_k= X_k r^k \cos (k \theta)$''. 
It is ${\rm Var}[S_N] = \sum_{k=0}^N r^{2k} v_k \cos^2(k\theta) \geq \frac{1}{2}V_N(r)$, hence
	\begin{align}
		P \left( |S_N| \leq t \sqrt{ {\rm Var}[S_N] } \right) \leq C \left( t + \frac{c K {\rm Var}[S_N] }{ ( {\rm Var}[S_N] )^{3/2} } \right) \leq C' \left( t + \frac{K}{\sqrt{ V_N(r)} } \right),
 	\end{align}	
where we have also used that $\mathbb E|Y_k-\mathbb E[Y_k|^3 \lesssim K r^{3k}|\cos(k\theta)|^3 K v_k$ for all $k$.
Note also that $|S_N| \leq |F_N(re^{i\theta})|$, hence we obtain
	\begin{align}
		P \left( |F_N(re^{i\theta})| \leq \frac{t}{\sqrt{2}} \sqrt{V_N(r)} \right) \leq C \sqrt{2} \left( t + \frac{K}{\sqrt{ V_N(r)}} \right).
	\end{align}
From the a.s. (uniform) convergence and Fatou's lemma we have 
	\begin{align}
		P \left( |F(re^{i\theta})| < \frac{t}{2} \sqrt{V(r)} \right) 
				\leq \liminf_N P \left( |f_N(re^{i\theta})| \leq \frac{t}{\sqrt{2}} \sqrt{V_N(r)} \right) \leq C' \left( t + \frac{K}{\sqrt{ V(r)}} \right).	
	\end{align}
The rest of the proof follows the argument of Theorem \ref{thm:s-nat-bd} mutatis mutandis. \prend

\begin{remark}
	As in Theorem \ref{thm:s-nat-bd}, the conditions we impose here are shift invariant, hence the same conclusion continues to hold for any perturbation
	$F+f$ by a deterministic $f(z) = \sum_{n=0}^\infty c_n z^n$ with $r_{F+f}=1$ a.s.
\end{remark}

\medskip

\noindent {\it Proof of Theorem \ref{thm:snb-be}.} Since $|X_n| \leq M$ a.s. we readily get $\limsup_n |X_n|^{1/n} \leq 1$ a.s. 
On the other hand, the boundedness and the divergence of the series $\sum_{n=0}^\infty {\rm Var}[X_n]$ in conjunction 
with Kolmogorov's 3-series theorem \cite{Gut-book} implies that the series $\sum_{n=0}^\infty X_n$ diverges a.s.
In particular, $r_F \leq 1$ a.s. Combining with the above we infer that $r_F=1$ a.s.

For the second assertion note that the assumptions of Proposition \ref{prop:3rd-var} are fulfilled since
	\[
		\mathbb E|X_n- \mathbb E[X_n]|^3 \leq 2M {\rm Var}[X_n],
	\]
for all $n$. The proof is complete. \prend

\subsection{Pointwise convergence on the circle of convergence}\label{pointwise}

One of the reasons for establishing the existence of a strong, and therefore an $L^\infty$-natural boundary,
is to gain information on the pointwise behavior of the power series on the circle of convergence. For example we have the 
following simple:

\begin{fact}  \label{q:uni-cvg-snb-1}
	Let $(c_k) \subset \mathbb C$ and suppose that the power series $f(z) = \sum_{k=0}^\infty c_k z^k$ has radius of convergence $r_f=1$. 
	If $f$ has $\mathbb T$ as $L^\infty$-natural boundary, then the following set 
	\[
		B _f := \left \{ z\in \mathbb T : \limsup_{N\to \infty} \left| \sum_{k=0}^N c_k z^k \right| = \infty \right\}
	\]
	is dense in the circle.   
\end{fact}

\noindent {\it Proof of Fact \ref{q:uni-cvg-snb-1}.} If not, there is a closed arc $J=\{z\in \mathbb T : \gamma \leq {\rm Arg}(z) \leq \delta\}$, ($0\leq \gamma < \delta <2\pi$)
for which $J\cap B = \emptyset$. That is, the 
sequence $f_N(\theta) :=\sum_{k=0}^N c_k e^{i k\theta}, \; N=1,2,\ldots$ is pointwise bounded for all $\theta \in [\gamma,\delta]$. Hence, the 
Osgood theorem \cite[p. 160, Theorem 32]{Roy} 
yields the existence of a closed sub-interval $I=[a,b]\subset [\gamma, \delta]$ on which $(f_N)$ is uniformly bounded.
Namely, there exists $M>0$ such that 
	\begin{align} \label{eq:ub}
		\sup_N \sup_{\theta \in I} \left| \sum_{k=0}^N c_k e^{i k\theta} \right| \leq M.	
	\end{align}
We have the following:

\smallskip

\noindent {\it Claim.} For the circular sector $S = \{ re^{i\theta} \mid 0\leq r < 1, \, \theta \in I\}$ we have $|f(z)| \leq M$ for all $z\in S$.

\smallskip

\noindent {\it Proof of Claim.} Let $z\in S$; there exist $\phi \in I$ and $0\leq t<1$ so that $z= t e^{i\phi}$. Let 
$N\in \mathbb N$ be arbitrary, but fixed. A direct application of Abel's inequality \cite[Problem 14.1]{Stee-cs-book}) yields
	\begin{align}
		|f_N(z)| = \left | \sum_{k=0}^N c_k t^k e^{i k \phi} \right| \leq  \max_{m\leq N}\left|\sum_{k=0}^m c_k e^{i k \phi}\right| \stackrel{\eqref{eq:ub} }\leq M. 
	\end{align}
Since $(f_N)$ converges to $f$ at each $z\in S$ we obtain $|f(z)| \leq M$ for all $z\in S$, as asserted. \prend

\medskip

With regard to random power series, Dvoretzky and Erd\H os in \cite{DE} gave a sufficient condition for which the {\it everywhere} divergence
of a Rademacher power series $\sum_{k=0}^\infty c_k \varepsilon_k z^k$ on $\mathbb T$ is a sure event. Their result reads as follows:
	\begin{quote}
		(Dvoretzky, Erd\H os, 1959). 
		\it Let $(t_k)$ be a monotone sequence of positive numbers with $t_k \downarrow 0$ and 
		$\limsup_{k} \left\{ \frac{1}{\log(1/t_k)} \sum_{j=0}^k t_j^2 \right\} >0$. 
		Then, for any sequence of complex numbers $(c_k)$ with $|c_k|\geq t_k$ for all $k$ we have that a.s. the Rademacher power series 
		$F(z)=\sum_{k=0}^\infty c_k \varepsilon_k z^k$ has $B_F=\mathbb T$.
	\end{quote}

For context let us consider the simple case $t_k = k ^{-1/2}$. Note that all conditions 
of Dvoretzky and Erd\H os's result are fulfilled, and hence the resulting Rademacher power
series $F$ diverges everywhere on the unit circle a.s. At the same time $\sum_{k=0}^\infty |c_k|^2 = \infty$, and Proposition \ref{prop:Rad-snb} yields that 
$F$ has $\psi$-(SNB). Although it might be compelling to link the presence of (SNB) with the everywhere divergence of $F$ on the boundary, this is simply not
the case. Quite remarkably, Michelen and Sawhney proved very recently \cite{MS} a sharp threshold phenomenon for the class of Rademacher random power series
affirming a conjecture of Erd\H os from \cite[Section V]{Erd-unsolv}. Their result asserts that if $(c_k)$ is a sequence of complex numbers satisfying $|c_k|= o(k^{-1/2})$, then 
for the Rademacher power series $F(z) = \sum_{k=0}^\infty c_k \varepsilon_k z^k$ the set 
$C_F: =\{z\in \mathbb T \; : \; \sum_{k=0}^\infty c_k \varepsilon_k z^k \; \textrm{converges} \}$ has almost surely Hausdorff dimension $1$.

Consequently, for $c_k = (k \log k)^{-1/2}$ we infer that the Rademacher power series $F$ may have $|z|=1$ as a (SNB) a.s, and yet the set of points of 
convergence on the boundary is large in a certain sense. 
These simple examples highlight how distinct the two singular behaviors (strong natural boundary vs divergence on the boundary) can be; 
one may only hope for mere implications from one to the other, 
and modified criteria should be investigated in the context of a random analytic function for establishing them.

\subsection{Local log-integrability} \label{S:log-int}

With the method we developed so far it's not clear if we can derive  that the integral $\int_I \log|F(re^{i\theta})| \, d\theta$ over any
arc $I$ blows up when $r\uparrow 1$. This is mainly due to the fact that $\log|F|$ may also take negative values. 
The significance of the quantity $\int_I \log|f(re^{i\theta})| \, d\theta$ for analytic functions $f$ on the disc is featured by its connection to the distribution
of roots of $f$ via the well-known Jensen formula \cite{AN}.

Here we show that one can establish local non log-integrability under some additional moment assumptions and slightly stronger anti-concentration properties 
on the r.v.s. of coefficients. The approach follows mutatis mutandis the reasoning presented in \cite[Section 4.2]{DV}, 
thus we only highlight the appropriate adjustments. The result reads as follows:

\begin{theorem} \label{thm:local-log-int}
	Let $\{X_k\}_{k=0}^\infty$ be independent random variables with $\mathbb E[X_k]=0$, $\sup_k \mathbb E[X_k^2]<\infty$, 
	and assume that $\sup_k Q(X_k, \lambda) \leq b \lambda$ for all $\lambda>0$ and some $b>0$.
Then, the random power series $F(z) = \sum_{k=0}^\infty X_k z^k$ has radius of convergence $r_F=1$, and for any 
arc $I \subset (0,2\pi)$ satisfies
		\begin{align}
			\sup_{0<r<1} \int_I \log|F(re^{i\theta})| \, d\theta = \infty \quad \textrm{a.s.}
		\end{align}
\end{theorem}
 
This result will follow from a more precise quantitative form for random polynomials (i.e., the partial sums of $F$). To this end, fix $N\in\mathbb N$ and
$X_0, \ldots, X_N$ independent r.v.s as in Theorem \ref{thm:local-log-int}. Let $F_N(z) = \sum_{k=0}^N X_k z^k$ be the partial sum of $F$ and let 
$W_z:= |F_N(z)|^2$. An argument similar to the one in \cite[Section 4.2]{DV} will yield

\begin{lemma}\label{f:4-5}
For all $z\in \mathbb C$ and all $t>0$ it holds that
\begin{align} \label{eq:4-1}
		P\left( \left|\, \log W_z - \log \mathbb E [W_z] \,\right| > t \right) \leq CK_b e^{-ct},
\end{align} 
where $K_b>0$ is a constant depending only on $b$.
\end{lemma}

Using the above we arrive at the following distributional inequality for the $\int_I \log|F_N(re^{i\theta})|\, d\theta$:

\begin{proposition} \label{prop:sector} 
	Let $\{X_k\}_{k=0}^N$ be independent r.v.s as in Theorem \ref{thm:local-log-int}. 
	Then for all $r>0, t>0$, and for any $I\subset (0,2\pi)$ the random polynomial $F_N(z) = \sum_{k=0}^N X_k z^k$ satisfies 	
	\begin{align} \label{eq:main-4-1}
			\mathbb P\left( \left| \fint_I \log |F_N(re^{i\theta})| \, d\theta -\frac{1}{2}\log \rho_N(r) \right| > t \right) \leq CK_{b} e^{-ct}, 
	\end{align} where $C,c>0$ are universal constants and $\rho_N(r)=\sum_{k=0}^N r^{2k}\mathbb E|X_k|^2$. 
\end{proposition}

\begin{proof} 
For any fixed $r>0$ and $I\subset (0,2\pi)$ we set
\begin{equation} 
      W_\theta := W_{re^{i\theta}}, \quad
      \quad 
      Y_\theta: = \log \left( \frac{W_\theta}{\mathbb E[W_\theta] }\right),
      \quad
	 Y_I := \frac{1}{| I |} \int_I Y_\theta  \, d\theta.
\end{equation}
In view of Lemma \ref{f:4-5}, we may apply \cite[Lemma 4.7]{DV} for $(\pm Y_\theta)_{\theta\in I}$ to get
	\begin{align} \label{eq:4-6}
		\mathbb P \left( |Y_I| > t \right) \leq 18 K_{b} e^{ - t/4}, \quad t>0.
	\end{align} 
It remains to notice that $\mathbb E[W_z] = \sum_{k=0}^N |z|^{2k} \mathbb E[X_k^2]$. \end{proof}

The next result can be viewed as the arc-wise counterpart of \cite[Lemma 4.11]{DV}. We couldn't locate a reference in the
classical literature on Complex Analysis, so we included a proof in the Appendix.

\begin{proposition} [convergence of log-integrals along arcs] \label{prop:cvg-arcs}
	Let $I=\{re^{i \theta} \mid a<\theta <b \}$, $r>0, \; (0\leq a < b < 2\pi)$ and let 
	$U\subset \mathbb C$ be open set with $\overline{I} \subset U$. Let $f_n, f : U \to \mathbb C$ be holomorphic functions on $U$ with $f\neq \bf 0$ such that 
	$f_n\to f$ uniformly on compact subsets of $U$ as $n\to \infty$. Then,  we have
		\begin{align} \label{eq:cvg-arc}
			\lim_{n\to \infty} \int_I \log |f_n|  = \int_I \log |f|.
		\end{align}
\end{proposition}

We are now ready to put everything together to obtain the main result of this section.

\medskip

\noindent {\it Proof of Theorem \ref{thm:local-log-int}.}  Let $F_N(\omega ; z) = \sum_{k=0}^N X_k(\omega) z^k$ be the $N$-th partial sum of $z\mapsto F(\omega; z)$. 
If we apply Proposition \ref{prop:sector} for $t=t_{N,r}= \frac{1}{4} \log \rho_N(r), \; 0 < r <1$, then for the random event
\begin{align}
		{\cal E}_N(I,r) := \left\{  \int_I \log|F_N(\omega, re^{i\theta} ) | \, d\theta <  \frac{| I |}{4} \log \rho_N(r) \right \},
	\end{align}
we conclude that
	\begin{align} \label{eq:4-9}
		\mathbb P( {\cal E}_N(I, r) ) \leq CK_b e^{-c \log \rho_N(r) }.
	\end{align}
Now Fatou's lemma implies that 
	\begin{align} \label{eq:4-9-b}
	\mathbb P \left( \liminf_N {\cal E}_N(I,r) \right) \leq \liminf_N \mathbb P( {\cal E}_N(I,r)) \stackrel{\eqref{eq:4-9}}\leq CK_b e^{-c\log \rho_\infty(r)},
	\end{align}
where $\rho_\infty(r) = \sum_{k=0}^\infty r^{2k}\mathbb E|X_k|^2$. Then Proposition \ref{prop:cvg-arcs} yields that 
	\begin{align}
		{\cal E}_\infty(I,r):=\left\{ \int_I \log|F(\omega, re^{i\theta} ) | \, d\theta <  \frac{| I |}{4} \log \rho_\infty (r) \right\} \subset \liminf_N {\cal E}_N (I,r).
	\end{align}
Notice that $\lim_{r \uparrow 1} \rho_\infty(r)=\infty$ in the light of \eqref{eq:cond-nb}, hence there exists a sequence $(r_k) \subset (0,1)$ so that 
$\rho_\infty(r_k) = e^k$. For the choice $r=r_k$ we further define the event ${\cal E}_I:= \limsup_k {\cal E}_\infty(I,r_k)$. An application
of the 1st Borel-Cantelli lemma completes the proof.  \prend

\section{Further questions and remarks}

As discussed in Section \ref{S:log-int}, anti-concentration assumptions on the coefficients may guarantee that the integral of
$\log|F|$ is infinite (on every arc) and, in turn 
, this yields the existence of infinitely many zeros of $F$ in the disc. In fact, in \cite[Theorem 4.12]{DV}
it is shown that under the assumptions $\{X_k\}$ independent, zero mean, variance one, and $Q(X_k, \varepsilon) \leq K\varepsilon$ for all $\varepsilon>0$,
$k=0,1,2,\ldots$ one can derive an asymptotically exact distributional information
for the roots of $F$. Namely, if $1/2<R_s<1$ is the largest radius for which $F$ has no more than
$s$ roots in the annular region $\Delta(0,1/2, R_s) = \{z\in \mathbb C : 1/2 < |z| \leq R_s\}$, then
\begin{align} \label{eq:dist-rt}
1-R_s \asymp \frac{\log s}{s}, \quad \textrm{ as} \; s\to \infty,
\end{align} with
probability one.\footnote{The estimate in \cite{DV} is a one-sided only estimate, as $R_s$ there stands for the largest $R$ for which $F$ has
no more than $s$ roots in the disc $D(0,R)$. However, a straightforward adaptation yields \eqref{eq:dist-rt}.}
It would be interesting to know if the (weaker) asymptotic anti-concentration condition \eqref{eq:cond-nb} (or \eqref{eq:cond-nb-2}) is sufficient for this phenomenon.
Regarding this, and taking into account the implication of Theorem \ref{thm:s-nat-bd} in the case of symmetric r.v.s,
we mention the following remarkable result of Nazarov, Nishry, and Sodin from \cite{NNS}: If $\{X_k\}$ are independent symmetric
r.v.s. with $\sum_{k=0}^\infty |X_k|^2 = \infty$ a.s. and the random power series $F(z) = \sum_{k=0}^\infty X_k z^k$ has radius of convergence 1 a.s., then
$F$ takes every complex value infinitely many times a.s. Namely, they prove the following Blaschke condition: With probability one, the random
power series $F$ satisfies
\begin{align*}
\forall w\in \mathbb C, \quad \sum_{\{z\in \mathbb D \, : \, F(z) = w\}} (1-|z|) = \infty.
\end{align*}
In particular, $F$ has infinitely many zeros in $\mathbb D$ and the latter sum provides a distributional information in the spirit of \eqref{eq:dist-rt}.

Further investigation of conditions that determine the distribution of zeros of random power series could benefit from some classical, 
deterministic results of Collingwood and Cartwright in \cite{CC?}. 
For example, for any $f$ holomorphic on $\mathbb D$, Collingwood and Cartwright show that if
\begin{equation} \label{cccondition}
     \sup_{0<r<1} \frac1{2\pi}
     \int_0^{2\pi}
     \log^+ |f(re^{i\theta})|d\theta = \infty,
\end{equation}
then any point in the Riemann sphere $\widehat{\mathbb{C}}$ is the limit of $f(z_n)$ for some $|z_n|\to 1$ \cite[Theorem 1]{CC?}.

Then \cite[Theorem 8]{CC?}  takes over to show that, given \eqref{cccondition}, the set of constant values along a sequence tending 
to the circle is dense in the Riemann sphere $\widehat{\mathbb{C}}$.
In particular, there is such sequence with constant value arbitrarily close to $0$.

Note here that, in general, one should not count on avoiding zeros by asking $f$ to blow up on the unit circle: 
If all curves approximating the circle have values going to infinity then any complex number (zero in particular) 
can be the constant value of $f$ along  some infinite sequence $(z_n) \subset \mathbb D$ with $|z_n|  \to 1$, cf.\ \cite[Theorem 9(ii)]{CC?}.

\smallskip

\bibliography{prob-snb-ref}
\bibliographystyle{alpha}


\appendix

\section{Proof of Fact \ref{fct:rev-wL2}}

The conclusion of Fact \ref{fct:rev-wL2} will follow from the following more general result for sub-gaussian r.v.s. Recall that 
a r.v. $X$ on $(\Omega, {\cal E}, P)$ is said to be sub-gaussian with constant $B>0$ if it has sub-gaussian tails, i.e., 
	\begin{align}
		P(|X| > t) \leq 2 \exp(-t^2/B^2), \quad t>0.
	\end{align}
The sub-gaussian constant $B>0$, in turn, can be quantified in terms of the Orlicz norm with Young function $\psi_2(t)= e^{t^2}-1$, $t\geq 0$ (cf. \cite[Section 2.5]{Ver-book}):
	\begin{align}
		\|X\|_{\psi_2} : = \inf \left \{ t>0 : \int_{\Omega} \psi_2(|X|/t) \, dP \leq 1 \right\}.
	\end{align}
It is a well known fact that bounded r.v.s comprise a sub-class of sub-gaussian r.v.s; more precisely we have the following:
\begin{fact} 
	\label{fct:A-bd-sg} If $X$ is bounded r.v., then $\|X\|_{\psi_2} \leq (\log 2)^{-1/2} \|X\|_\infty$.
\end{fact}
For related definitions, results, as well as a proof of the latter fact, 
the reader is referred to \cite[Section 2.5]{Ver-book}. In light of the above we shall prove the following:

\begin{proposition} \label{prop:rev-wL2-sg}
	Let $X$ be a r.v. which is sub-gaussian. Then, we have
		\begin{align} \label{eq:rev-wL2-b} 			
			{\rm Var} [X] \leq C \log\left( \frac{e\|X\|_{\psi_2}}{{\sqrt{ {\rm Var} [X] }}}\right) \cdot \sup_{\delta>0} \left\{ \delta^2 [1-Q(X, \delta)]\right\},
		\end{align}
	where $C>0$ is a universal constant.
\end{proposition}

Note that Fact \ref{fct:rev-wL2} is an immediate consequence of the latter result owing to the Fact \ref{fct:A-bd-sg}. 
For proving Proposition \ref{prop:rev-wL2-sg}, 
first we show an equivalent expression for the quantity $\delta^2[1-Q(X,\delta)]$ in terms of the distribution function. To this end, recall the 
weak $L^2$ norm of a measurable function $f$ on some measure space $(X, {\mathcal A} , \mu)$:
	\begin{align}
		\|f\|_{2,\infty} = \sup_{t>0} \left\{ t \mu(x\in X : |f(x)|>t)^{1/2} \right\}.
	\end{align}
It is known that the aforementioned quantity is not a norm, but equivalent to a norm (see \cite[Theorem 1.2.10]{Graf-FFA} for details). With this definition we have the following:

\begin{lemma} \label{lem:Levy-wL2}
	Let $X$ be a r.v. on some probability space with $(\Omega , {\mathcal E}, P)$ and let $M={\rm med}(X)$ be a median of $X$. Then, we have
		\begin{align}
			\sup_{\delta>0} \left\{ \delta^2 [1 - Q(X, \delta)] \right\} \asymp \|X - M\|_{2,\infty}^2.
		\end{align}
\end{lemma}

\noindent {\it Proof.} Let $\delta>0$. Then, we may write
	\begin{align}
		\delta^2[1-Q(X,\delta)] \leq \delta^2 P(|X-M|>\delta/2) \leq 4\|X-M\|_{2,\infty}^2.
	\end{align}
For the reverse estimate we will need the following probabilistic fact:

\begin{fact} [Weak Symmetrization] \label{fct:A2}
	Let $\xi$ be a r.v. on $(\Omega, {\mathcal E}, P)$. If $\xi'$ is an independent copy of $\xi$, then for all $t>0$ we have
		\begin{align}\label{fct:A2}			
			\frac{1}{2} P(|\xi-{\rm med}(\xi)| >t) \leq P(|\xi -\xi'|>t) \leq 2 \inf_{v\in \mathbb R} P(|\xi - v|>t/2).
		\end{align}
\end{fact} 

For a proof of this fact the reader is referred to \cite[Proposition 2.6]{Gut-book}. Using \eqref{fct:A2} we may write
	\begin{align}
		1-Q(X,\delta) =1-\sup_{v\in \mathbb R} P(|X-v| \leq \delta/2) = \inf_{v\in \mathbb R} P(|X-v|>\delta/2) \stackrel{\eqref{fct:A2}}\geq \frac{1}{4}P(|X-M|>\delta).
	\end{align}
We derive that $\delta^2 [1-Q(X,\delta)] \geq \frac{\delta^2}{4} P(|X-M|>\delta)$ which implies that 
	\begin{align}
		\sup_{\delta>0} \left\{ \delta^2 \left[1-Q(X,\delta) \right]\right\} \geq \frac{1}{4} \|X-M\|_{2,\infty}^2,
	\end{align}
as claimed. \prend

\medskip

\noindent {\it Proof of Proposition \ref{prop:rev-wL2-sg}.} Recall that ${\rm Var}[X] = \inf_{v\in \mathbb R} \mathbb E|X-v|^2$. 
Also, by the definition of the $\psi_2$ norm we have the inequalities 
	\begin{align}
	\|X\|_{2,\infty} \leq \|X\|_2 \leq \|X\|_{\psi_2}, \quad |M| \leq \mathbb E|X| + \sqrt{ {\rm Var}[X] }\leq 2 \|X\|_{\psi_2},
	\end{align}
where $M = {\rm med}(X)$. If we set $Y:=X-M$ we obtain $B:=\|Y\|_{\psi_2}\leq 3\|X\|_{\psi_2}$. Hence, for $0<\lambda<1$ (to be determined later)  we may write
	\begin{align*}
		{\rm Var}[X] &\leq 2\int_0^{\infty} t P(|Y|>t) \, dt \leq  4 \int_0^\infty t [P(|Y|>t)]^{1-\lambda} e^{-\lambda t^2/B^2} \, dt \\
					& \leq 4 \|Y\|_{2,\infty}^{2(1-\lambda)} \int_0^\infty t^{2\lambda-1} e^{-\lambda t^2/B^2} \, dt \\
					& = 2 \|Y\|_{2,\infty}^{2(1-\lambda)} \left( \frac{B}{\sqrt{\lambda}} \right)^{2\lambda} \int_0^\infty s^{\lambda-1} e^{-s} \, ds \\
					& \leq 2 e^{1/e}\|Y\|_{2,\infty}^2 \left( \frac{B}{\|Y\|_{2,\infty}} \right)^{2\lambda} \Gamma(\lambda),
	\end{align*}
where we have applied a change of variable, and the elementary inequality $\lambda^\lambda \geq e^{-1/e}$ for all $\lambda\in (0,1)$. Note that for 
all $\lambda\in (0,1)$ we have\footnote{It is $\lambda \Gamma(\lambda) = \Gamma(1+\lambda) \leq 1$ for all $0<\lambda <1$, since 
$\Gamma$ is log-convex and $\Gamma(1) = \Gamma(2)=1$.} $\Gamma(\lambda) \leq 1/ \lambda$, hence
	\begin{align}
		{\rm Var}[X] \leq 2e^{1/e} \|Y\|_{2,\infty}^2 \frac{1}{\lambda} \left( \frac{B^2}{\|Y\|_{2,\infty}^2}\right)^\lambda. 
	\end{align}
The choice $1/ \lambda = 1 + \log(B/\|Y\|_{2,\infty}) \geq 1$ yields
	\begin{align}
		{\rm Var}[X] \leq 2e^{2+1/e} \|Y\|_{2,\infty}^2 \log(e B / \|Y\|_{2,\infty}).
	\end{align}
It follows that\footnote{Let $\beta = {\rm Var}[X]$, $\alpha = \|Y\|_{2,\infty}^2$, and $B = \|Y\|_{\psi_2}^2$. We have shown that 
$\beta \leq e^4\alpha \left(1+ \log(B/\alpha) \right)$, so for $x=B/\alpha$ and $y=B/\beta$ that is $x \leq e^4 y(1+\log x)$ with $x\geq 1$. The latter implies
that $x\leq 2e^4 y(\frac{1}{2}+\log\sqrt{x}) \leq 2e^4 y(\sqrt{x} - 1/2) < e^5y\sqrt{x}$ or $x\leq e^{10} y^2$. Going back to the previous we find $x\leq e^4 y (11+2\log y)$.} 
\begin{align}
	{\rm Var}[X] \leq e^4 \|Y\|_{2,\infty}^2 \left[ 11 + 2 \log\left( \frac{B^2}{{\rm Var}[X]} \right) \right].
\end{align}
Taking into account Lemma \ref{lem:Levy-wL2} we conclude \eqref{eq:rev-wL2-b} as claimed. \prend

\begin{note}
	The family of r.v.s $\{X_\alpha\}$ constructed in Example \ref{ex:vwa-opt} shows that estimate \eqref{eq:rev-wL2} is optimal (up to constants):
		\begin{align}
			{\rm Var}[X_\alpha] = \|X_\alpha\|_2^2 \asymp \alpha |\log\alpha | \asymp \|X_\alpha\|_{2,\infty}^2 \log(\frac{\|X_\alpha\|_\infty}{\|X_\alpha\|_2}), \quad \alpha \to 0^+.
		\end{align}
\end{note}

\section{Log-integration over arcs}

Purpose of this appendix is to give a proof of Proposition \ref{prop:cvg-arcs}. The following result extends \cite[Lemma 4.11]{DV}.

\begin{theorem} [Convergence of integrals along arcs] \label{thm:B1}
	Let $I=\{re^{i \theta} \mid a<\theta <b \}$, $r>0, \; (0\leq a < b < 2\pi)$ and let 
	$U\subset \mathbb C$ be open set with $\overline{I} \subset U$. Let $f_n, f : U \to \mathbb C$ be holomorphic functions on $U$ with $f\neq \bf 0$ such that 
	$f_n\to f$ uniformly on compact subsets of $U$ as $n\to \infty$. Then,  we have
		\begin{align} \label{eq:cvg-arc}
			\lim_{n\to \infty} \int_I \log |f_n|  = \int_I \log |f|.
		\end{align}
\end{theorem}

A special case is that of a sequence of monomials which converges to a monomial with a root on the arc. Since this case is 
instructive, but also necessary for the generic case later, we present it here separately as an auxiliary lemma.

\begin{lemma} \label{lem:cvg-arcs-0}
	Let $U\subset \mathbb C$ be an open set. Let $I=\{re^{i\theta} \mid a\leq \theta \leq b\}$, $(0\leq a<b \leq 2\pi)$ and $z_0\in U$ so that $z_0\in I \subset U$. 
	Then, for any sequence $(z_n) \subset U$ with $z_n\to z_0$ as $n\to \infty$, we have
		\begin{align}
			\int_a^b  \log|re^{i\theta} - z_n| \, d\theta \to \int_a^b \log |re^{i\theta} - z_0| \, d\theta, \quad n\to \infty.
		\end{align}
\end{lemma}

\smallskip

\noindent {\it Proof.} Let $\mu$ denote the arc-length measure on $I$ and let $U_\mu$ be the corresponding logarithmic potential, i.e. $U_\mu(z) = \int \log| z-w|\, d\mu(w) $.
In this framework, the assertion asks for the continuity of $U_\mu$ at $z_0$. To this end, we employ the following fact \cite[Theorem 3.1.3]{Ra-book}:

\begin{fact} [Continuity principle] \label{fact:cont-prin}
	Let $\mu$ be a Borel measure supported on a compact set $K\subset \mathbb C$. The logarithmic potential $U_\mu(z) = \int \log |z-w| \, d\mu(w)$ is continuous if and only
	if is continuous when restricted on $K$, that is 
		\begin{align}
			\lim_{z\to z_0} U_\mu(z) =\lim_{z\to z_0 \atop z\in K} U_\mu(z).
		\end{align}
\end{fact}

In the light of Fact \ref{fact:cont-prin} is suffices to consider the case that $(z_n)\subset I$. Let $z_0=re^{i t_0}$, $a\leq t_0\leq b$ and let 
$z_n = r e^{i t_n}$ where $a\leq t_n \leq b$ and $t_n \to t_0$. Then, we may write
	\begin{align*}
		U_\mu (z_n) &= \int_a^b \log |re^{i\theta} -re^{i t_n}| \, d\theta = \int_a^b \log |re^{i(\theta +t_0- t_n)} -re^{i t_0}|\, d\theta \\
					&= \int_{a+t_0-t_n}^{b+t_0-t_n} \log |re^{i\phi} - re^{it_0}|\, d\phi \to  \int_a^b \log |re^{i\phi} - re^{it_0}| \, d\phi = U_\mu(z_0),  
	\end{align*}
as required. \prend

\medskip

\noindent {\it Proof of Theorem \ref{thm:B1}.} We distinguish the following cases.

\medskip

\noindent {\bf Case I:} The closed arc $\overline I$ is zero-free set for $f$, i.e., $Z_f \cap \overline {I} = \emptyset$.  

\smallskip

Since $f$ is continuous, there exists $\delta>0$ such that $|f(z)|\geq \delta$ for all $z\in \overline {I}$.
The uniform convergence shows that there exists $n_0=n_0(\delta) \in \mathbb N$ such that $|f_n(z)| \geq \delta/2$ for all $z\in \overline{I}$ 
and for all $n\geq n_0$. It  follows that $\log |f_n| \to \log |f|$ uniformly on the compact set $ \overline{I}$ which contains $I$, 
thus \eqref{eq:cvg-arc} readily follows.

\medskip

\noindent {\bf Case II:} The arc $\overline {I}$ contains roots of $f$ -- 
say $Z_f \cap \overline{I} = \{z_1, \ldots, z_p \}$, where $z_j$ comes with multiplicity $m_j$, $j=1,2,\ldots,p$. 

\smallskip

This case is more involved since one has to show that the sequence of roots $(z_j^{n})_{j=1}^p$ 
of $f_n$ do not cause any obstruction to the convergence of the integrals despite the unboudedness of the integrands $\log|z-z_j|$. 
We proceed with the details. There exists\footnote{e.g., $0<r< \frac{1}{2} \min\{{\rm dist}(\overline{I}, U^c),  |z_k-z_\ell | : 1\leq k < \ell \leq p \}$.} $r>0$ 
such that $\overline{D(z_k, r)} \subset U$, $\overline{D(z_k, r)} \cap \overline {D(z_\ell, r)} = \emptyset$ for $k\neq \ell$, 
and $Z_f \cap \overline{D(z_k, r)} = \{z_k\}$ for $k=1,\ldots,p$. By Hurwitz's theorem \cite[Chapter 5]{AN} there exists $N_0\in \mathbb N$ such that
$|Z_f \cap D(z_k, r)| = m_k$ for all $n\geq N_0$ and $k=1,\ldots,p$. We set $I_k := \overline{I} \cap \overline{D(z_k, r/2)}$ and $I^\ast : = I \setminus \bigcup_{k=1}^p I_k$. Next, 
we write
	\begin{align}
		\int_I \log |f_n| = \int_{I^\ast} \log |f_n| + \sum_{k=1}^p \int_{I_k} \log |f_n|.
	\end{align}
Note that $I^\ast$ is finite union of non-overlapping arcs and $f$ is zero-free on $\overline{I^\ast}$. Hence, for the first integral we may apply Case I component-wise.
It remains to establish the convergence for the integrals over the arcs containing single roots. We isolate this case in the following technical lemma:

\begin{lemma} \label{lem:cvg-arc-rt}
	Let $z_0\in \mathbb C$, $\rho>0$ and let $h_n,h: D(z_0, \rho) \to\mathbb C$ with $h_n\to h$ uniformly on compact subsets of $D(z_0, \rho)$. Suppose that 
	$z_0$ in the unique root of $h$ in $D(z_0, \rho)$. Then, for any closed arc $I$ so that $z_0 \in I \subset D(z_0, \rho)$ we have
		\begin{align}
			\int_I \log|h_n| \to \int_I \log |h|, \quad n\to \infty.  
		\end{align}
\end{lemma}

\noindent {\it Proof of Lemma \ref{lem:cvg-arc-rt}.} Let $\rho/2 < \delta < \rho$ so that $I\subset D(z_0, \delta)$. 
Hurwitz's theorem \cite[Chapter 5]{AN} yields the existence of $n_0\in \mathbb N$ so that $|Z_{h_n} \cap D(z_0, \delta) | =\{z_{n,1},\ldots, z_{n,m_0}\}$\footnote{Here the zero-set of $h_n$ is viewed as multiset, i.e., $z_{n,j}$ are not necessarily distinct.} for all $n\geq n_0$, 
where $m_0$ is the multiplicity of $z_0$, and $z_{n,j} \to z_0$ as $n\to \infty$ for $j=1,2,\ldots,m_0$. Next, we recall the following standard fact \cite[p.\ 22]{AN}:
\begin{fact}
	If $\phi :D(z_0, \rho) \to \mathbb C$ is holomorphic function with roots $w_1, \ldots, w_m$ (not necessarily distinct) then there exists $\phi_0: D(z_0, \rho) \to \mathbb C$
	holomorphic with $\phi_0(w) \neq 0$ for all $w\in D(z_0, \rho)$ and
		\begin{align}
			\phi(w) = \phi_0(w) \prod_{j=1}^m (w-w_j), \quad w\in D(z_0, \rho).
		\end{align}
\end{fact}

Applying the fact for each $h,h_n$ we find holomorphic functions $h_0, h_{n,0}: D(z_0,\rho) \to \mathbb C$ so that 
\begin{align}
	h(z) = (z-z_0)^{m_0} h_0(z), \quad h_n(z) = \prod_{j=1}^{m_0}(z-z_{n,j}) h_{n,0}(z), \quad h_0(z), h_{n,0}(z) \neq 0, \; z\in \overline{D(z_0,\delta)}.
\end{align}

Thereby, we may write
	\begin{align}
		\int_I \log |h_n| = \sum_{j=1}^{m_0} \int_I \log |z-z_{n,j}| d\mu(z) + \int_I \log |h_{n,0}|,
	\end{align}
where $\mu$ stands for the arc-length measure on $I$. In view of Lemma \ref{lem:cvg-arcs-0} we have 
	\begin{align} 
		\int_I \log |z-z_{n,j}| \, d\mu(z) \to \int_I \log |z-z_0| \, d\mu(z), \quad j=1,2,\ldots,m_0. 
	\end{align}
For the other integral, since $h_0(z)\neq 0$ for $z\in\overline{ D(z_0,\delta)}$, 
it suffices to show that $h_{n,0} \to h$ uniformly on $\overline {D(z_0, \delta)}$. Clearly, $h_{n,0} \to h_0$ uniformly
on $\overline{D(z_0,\delta)} \setminus D(z_0, \rho/2)$. It suffices to show that $h_{n,0}\to h_0$ uniformly on $\overline{D(z_0, \rho/2)}$. To this 
end, we invoke Cauchy's integral theorem \cite[Section 2.2]{AN}: There exists $N_0\in \mathbb N$ so that $\max_{j\leq m_0} |z_{n,j}-z_0| \leq \rho/2$ for all $n\geq N_0$,
and for any $|z-z_0| \leq \rho/2$ we may write
	\begin{align*}
		\left| h_{n,0}(z) -h_0(z) \right| \leq  \frac{1}{2\pi } \int_{C(z_0,\delta)} \left| \frac{ h_{n,0}(\zeta)-h_0(\zeta) }{\zeta-z} \right| \, d\zeta 
			\leq \frac{\delta}{\delta -\rho/2} \| h_{n,0} -h_0\|_{C(z_0,\delta)}.
	\end{align*}
The claim follows by the uniform convergence $h_{n,0} \to h_0$ on $C(z_0,\delta)$. The proof of Lemma \ref{lem:cvg-arc-rt} is complete. \prend


\end{document}